\documentclass{amsart}
\usepackage{a4}
\usepackage{amssymb}
\usepackage{amsmath}
\usepackage{amsthm}
\usepackage{enumerate}
\usepackage{mathrsfs}  

\usepackage{stmaryrd}

\def\ombb{\mathord{\mbox{\hspace*{0pt}}\;\!\overline{\!\!\:\omega\!\!\:}\;\!}\mbox{\hspace*{0pt}}}

\def\squareforqed{\hbox{\rlap{$\sqcap$}$\sqcup$}}
\def\qed{\ifmmode\squareforqed\else{\unskip\nobreak\hfil
\penalty50\hskip1em\null\nobreak\hfil\squareforqed
\parfillskip=0pt\finalhyphendemerits=0\endgraf}\fi\medskip}

\def\ovphantom{\mathord{\mbox{\hspace*{0pt}}\,\overline{\!\phantom{x}\!\!\;}\!\:}\mbox{\hspace*{0pt}}}
\def\ovA{\mathord{\mbox{\hspace*{0pt}}\,\overline{\!A\!\!\;}\!\:}\mbox{\hspace*{0pt}}}
\def\ovB{\mathord{\mbox{\hspace*{0pt}}\,\overline{\!B\!\!\;}\!\:}\mbox{\hspace*{0pt}}}
\def\ovC{\mathord{\mbox{\hspace*{0pt}}\,\overline{\!C\!\!\;}\!\:}\mbox{\hspace*{0pt}}}
\def\ovD{\mathord{\mbox{\hspace*{0pt}}\,\overline{\!D\!\!\;}\!\:}\mbox{\hspace*{0pt}}}
\def\ovE{\mathord{\mbox{\hspace*{0pt}}\,\overline{\!E\!\!\;}\!\:}\mbox{\hspace*{0pt}}}
\def\ovF{\mathord{\mbox{\hspace*{0pt}}\,\overline{\!F\!\!\;}\!\:}\mbox{\hspace*{0pt}}}
\def\ovM{\mathord{\mbox{\hspace*{0pt}}\,\overline{\!M\!\!\;}\!\:}\mbox{\hspace*{0pt}}}
\def\ovN{\mathord{\mbox{\hspace*{0pt}}\,\overline{\!N\!\!\;}\!\:}\mbox{\hspace*{0pt}}}

\def\ovu{\mathord{\mbox{\hspace*{0pt}}\,\overline{\!u\!\!\;}\!\:}\mbox{\hspace*{0pt}}}

\def\ovx{\mathord{\mbox{\hspace*{0pt}}\,\overline{\!x\!\!\;}\!\:}\mbox{\hspace*{0pt}}}
\def\ovy{\mathord{\mbox{\hspace*{0pt}}\,\overline{\!y\!\!\;}\!\:}\mbox{\hspace*{0pt}}}
\def\ovz{\mathord{\mbox{\hspace*{0pt}}\,\overline{\!z\!\!\;}\!\:}\mbox{\hspace*{0pt}}}

\def\ombb{\mathord{\mbox{\hspace*{0pt}}\;\!\overline{\!\!\:\omega\!\!\:}\;\!}\mbox{\hspace*{0pt}}}

\def\cn{\mathord{{\;\!{:}\;\!}}}
\def\EE{\mathord{\mathrm{E}}}
\def\CC{\mathord{\mathrm{C}}}
\def\NN{\mathord{\mathrm{N}}}
\def\II{\mathord{\mathrm{I}}}
\def\QQ{\mathord{\mathrm{Q}}}

\def\OO{\mathord{\mathbb{O}}}

\def\Orb{\operatorname{\mathcal{O}}}

\newcommand{\Fq}{\mathbb{F}_q}

\renewcommand{\bar}{\overline}
\newcommand{\omg}{\overline{\omega}}
\newcommand{\xone}{$e_{-1}$}
\newcommand{\xtwo}{$e_{\ombb}$}
\newcommand{\xthree}{$e_{\omega}$}
\newcommand{\xfour}{$e_{0}$}
\newcommand{\xfive}{$e_{-0}$}
\newcommand{\xsix}{$e_{-\omega}$}
\newcommand{\xseven}{$e_{-\ombb}$}
\newcommand{\xeight}{$e_{1}$}
\newcommand{\m}{$-$}
\newcommand{\im}{\mathrm{Im}\,}

\newcommand{\diag}{\mathrm{diag}}

\newcommand{\SL}{\mathrm{SL}}
\newcommand{\GO}{\mathrm{GO}}
\newcommand{\CGO}{\mathrm{CGO}}
\newcommand{\SO}{\mathrm{SO}}

\newcommand{\SE}{\mathrm{SE}}
\newcommand{\E}{\mathrm{E}}

\newcommand{\Spin}{\mathrm{Spin}}
\newcommand{\Tr}{\operatorname{T}}

\newcommand{\J}{\mathbb{J}}

\renewcommand{\bar}{\overline}
\renewcommand{\hat}{\widehat}

\renewcommand{\ker}{\operatorname{ker}}
\newcommand{\T}{\top}

\newcommand{\fdet}{\Delta}

\newcommand{\inner}[2]{\langle #1, #2 \rangle}

\makeatletter
\renewcommand*\env@matrix[1][c]{\hskip -\arraycolsep
  \let\@ifnextchar\new@ifnextchar
  \array{*\c@MaxMatrixCols #1}}
\makeatother

\usepackage{marginnote}
\usepackage{marginfix}

\newtheorem{theorem}{Theorem}
\numberwithin{theorem}{section}
\newtheorem{lemma}[theorem]{Lemma}
\newtheorem{corollary}[theorem]{Corollary}
\newtheorem{definition}[theorem]{Definition}
\newtheorem{proposition}[theorem]{Proposition}
\newtheorem*{remark}{Remark}

\newcommand{\Mod}[1]{\ (\mathrm{mod}\ #1)}

\title
[Octonions, Albert vectors and $\E_6(F)$]
{Octonions, Albert vectors and the group $\E_6(F)$}
\author{John~N.~Bray,\ Yegor~Stepanov,\ Robert~A.~Wilson}
\address{}
\email{johnnbray@hotmail.com, yegor@ystepanoff.net, r.a.wilson@qmul.ac.uk}
\begin{document}
\maketitle

\begin{abstract}
We present a uniform approach to the construction of the groups of type 
$\EE_6$ over arbitrary fields without using Lie theory. This gives a simple 
description of the group generators and some of the subgroup structure. 
In the finite case our approach also permits relatively straightforward 
computation of the group order. \newline
\end{abstract}
\tableofcontents
\section{Introduction}
The construction of the groups of type $\EE_6$ goes back to the work of Dickson
\cite{Dickson1, Dickson2}. He constructed the analogue of the complex Lie group $\EE_6$
as a linear group in $27$ variables which leaves a certain cubic form invariant. 
Jacobson, inspired by Dickson and by Chevalley's T\^{o}hoku paper \cite{ChevTohoku},
studied the automorphism group of an Albert algebra (this algebra consists of the 
$3\times 3$ Hermitian matrices written over an octonion algebra), 
and the stabiliser of the determinant
over the fields of characteristic not $2$ or $3$ in a series of papers \cite{Jac1, Jac2, 
Jac3}. For instance, he proved that if the Albert algebra contains nilpotent elements,
then the group is simple. It must have been implicit that the determinant of the elements
in the Albert algebra is essentially the same as Dickson's cubic form, although Jacobson
does not refer to Dickson. Moreover, although cases of characteristic $2$ and $3$ 
were of no problem to Dickson, they were problematic in Jacobson's construction. 
The series of papers by Aschbacher \cite{Asch1, Asch2, Asch3, Asch4, Asch5} also addresses the question
of constructing the groups of types $\EE_6$ and ${}^2\EE_6$ without mentioning Albert algebras or octonions at all. It is to be emphasised that Aschbacher's 
construction is of a rather abstract nature and some of the structural aspects require further 
research. In a preprint by R.~A.~Wilson \cite{WilsonPaper} the construction of
finite simple groups $\EE_6(q)$, $\mathrm{F}_4(q)$ and ${}^2\EE_6(q)$ is sketched.  

In the late 1980s the problem of classifying the maximal subgroups came into prominence. Aschbacher's 
study of the $27$-dimensional module for $\EE_6$ reveals much more structure rather than 
the standard $78$-dimensional representation. However, Aschbacher does not give a complete
list of maximal subgroups, which means there is still a need for a modern review of Dickson's 
construction.

\section{Some properties of $\Omega_{2m}(F, Q)$}
\label{section:orthogonal}

Let $V$ be a vector space of dimension $n$ over a field $F$ . We assume that there is a 
non-singular quadratic form $Q$ defined on $V$. Denote by $\GO_n(F,Q)$ the group of 
non-singular linear transformations that preserve the
form $Q$. In case $F$ is of characteristic $2$ we define the \textit{quasideterminant}
$\operatorname{qdet}:\GO_n(F,Q) \rightarrow \mathbb{F}_2$ to be the map
\begin{equation}
    \operatorname{qdet}: g \mapsto \dim_F ( \im(\II - g) )\ \operatorname{mod}\ 2.
\end{equation}
Further, the group $\SO_n(F,Q)$ is the kernel of the (quasi)determinant map. Finally,
the subgroup $\Omega_n(F,Q)$ of $\SO_n(F,Q)$ is defined as the kernel of another 
invariant called the spinor norm. If the characteristic of the field is not $2$, 
there exists a double cover of $\Omega_n(F,Q)$, denoted as $\Spin_n(F,Q)$. 

This section is devoted to some of the private life of the group $\Omega_{2m}(F,Q)$, which 
will be crucial in our further constructions. Consider the vector space $V$ of dimension
$2m+2$ over $F$ with a non-singular quadratic form $Q$ defined on it. 
Assuming that the Witt index of $Q$ is at least $1$, we can pick 
the basis $\mathcal{B} = \{v_1, w_1,\ldots,w_{2m}, v_2\}$ in $V$ such that $(v_1,v_2)$ 
is a hyperbolic pair. Consider the decomposition $V = \langle v_1 \rangle \oplus
\langle w_1, \ldots, w_{2m} \rangle \oplus \langle v_2 \rangle$ and denote 
$W = \langle w_1, \ldots, w_{2m} \rangle$. Further, denote by $Q_W$ the restriction of 
$Q$ on $W$.

\begin{lemma}
    \label{lemma:1_stabiliser_omega}
	The stabiliser in $\Omega_{2m+2}(F,Q)$ of the vector $v_1$ 
	is a subgroup of shape $W\cn\Omega_{2m}(F,Q_W)$, and the stabiliser of 
	the pair $(v_1,v_2)$ is a subgroup $\Omega_{2m}(F,Q_W)$. 
    
    %Suppose $\hat{Q}$ is a non-singular quadratic form of Witt index at least $1$ defined 
    %on $F^{2m+2}$. The stabiliser in $\Omega_{2m+2}(F,\hat{Q})$ of an isotropic vector is 
    %a subgroup of shape $F^{2m} \cn \Omega_{2m}(F,Q)$, where $Q$ is the restriction of 
    %$\hat{Q}$ on $F^{2m}$. The stabiliser of two vectors forming a hyperbolic pair is a
    %subgroup $\Omega_{2m}(F,Q)$.  
\end{lemma}

Note that any element in $\Omega_{2m+2}(F,Q)$ which fixes $v_1$ also stabilises
$\langle v_1 \rangle^{\perp}$, so with respect to the basis $\mathcal{B}$ it has the 
following matrix form:
\begin{equation}
	    \hat{A} = \left[
            \begin{array}{c|c|c}
                1 & 0 & 0\  \\ \hline
                 & & \\
                u_2^\T &\ \ \ \ \ A\ \ \ \ \  & 0\  \\
                 & & \\ \hline
                \mu & u_1 & \lambda\
            \end{array}
        \right],
\end{equation}
where $\lambda \in F$, $u_1, u_2 \in W$, and the matrix $A$ acts on $W$ as an element of $\Omega_{2m}(F,Q_W)$. Let $f$ be the 
polar form of $Q$. Then an element $\hat{A}$ in the stabiliser of $v_1$ 
acts on $v_2$ as
    \begin{equation}
        \begin{array}{r@{\;\;}c@{\;\;}l}
            v_2 & \mapsto & (\mu \mid u_1 \mid \lambda).
        \end{array}
    \end{equation}
The bilinear form $f$ is preserved, so we get
    \begin{equation}
        1 = f(v_1,v_2) = \mu f(v_1,v_1) + f(v_1, (0\mid u_1 \mid 0)) + \lambda f(v_1, v_2)
        = \lambda,
    \end{equation}
and hence $\lambda = 1$. 

Since $(v_1,v_2)$ is a
    hyperbolic pair, the form $f$ on $V$ can be represented by the Gram matrix
    \begin{equation}
        [f]_{\mathcal{B}} = \left[
            \begin{array}{c|c|c}
                0 & 0 & 1  \\ \hline
                 & & \\
                0 &\ \ \ \ \ B\ \ \ \ \  & 0 \\
                 & & \\ \hline
                1 & 0 & 0
            \end{array}
        \right],
    \end{equation}
where $B$ is the Gram matrix of $f_W$, the restriction of $f$ on $W$. 
We explore the fact that an element in the stabiliser of $v_1$ preserves the
    form $f$:
    \begin{equation}
        [f]_{\mathcal{B}} = \hat{A} \cdot [f]_{\mathcal{B}} \cdot \hat{A}^{\T} = 
        \left[
            \begin{array}{c|c|c}
                0 & 0 & 1\ \\ \hline
                & & \\
                0 &\ \ \ \ ABA^\T\ \ \ \ & ABu_1^\T + u_2^\T \\
                & & \\ \hline
                %& & \\
                1 & u_2+u_1 BA^\T & 2\mu + u_1 B u_1^\T
            \end{array}
        \right].
    \end{equation}
It follows that $u_2 = -u_1BA^\T$. We also have
\begin{multline}
	0 = Q(v_2) = Q(v_2 \hat{A}) = Q( (\mu \mid u_1 \mid 1) ) = \\
        = Q( (\mu \mid u_1 \mid 0) ) + Q(v_2) + f( (\mu\mid u_1 \mid 0), v_2)
        = Q_W(u_1) + \mu,
\end{multline}
so $\mu = -Q_W(u_1)$. As a result, the general element $\hat{A}$ in the stabiliser of
$v_1$ has the following form:
\begin{equation}
		    \hat{A} = \left[
            \begin{array}{c|c|c}
                1 & 0 & 0\  \\ \hline
                 & & \\
                -AB^{\T}u_1^{\T} &\ \ \ \ \ A\ \ \ \ \  & 0\  \\
                 & & \\ \hline
                -Q_W(u_1) & u_1 & 1\
            \end{array}
        \right].
\end{equation}

Witt's lemma tells us that the group $\GO_{2m}(F,Q_W)$ acts transitively 
on the set of non-zero vectors of each norm in $W$. In fact,
the same is true for $\Omega_{2m}(F,Q_W)$ in case when $Q_W$ 
is of Witt index at least $1$. 

\begin{lemma}
    \label{lemma:1_omega_transitive}
    The group $\Omega_{2m}(F,Q_W)$, where $Q_W$ is of Witt index at least $1$, 
    acts transitively on
    \begin{equation*}
	O_{\lambda} = \left\{ v \in W \ \big| Q_W(v) = \lambda,\ v \neq 0\right\}
    \end{equation*}
    for each value of $\lambda \in F$.
\end{lemma}

From now on we require that $Q_W$ has Witt index at least $2$. The following technical 
results will be our main tool in the construction of certain orthogonal subgroups
of $\E_6(F)$.

\begin{theorem}
    \label{theorem:1_omega_maximal}
    Let $Q_W$ be of Witt index at least $2$. 
    The subgroup $\Omega_{2m}(F,Q_W)$ is maximal in $W\cn\Omega_{2m}(F,Q_W)$. 
\end{theorem}

\begin{proof}
    Recall that $v_2 \in V$ is mapped under the action of 
    $\mathcal{G} = W\cn\Omega_{2m}(F,Q_W)$ to a vector of the
    form $(-Q_W(u) \mid u \mid 1)$, where $u$ is an element of $W$. 
    Since the stabiliser of $v_2$ in $\mathcal{G}$ is $\Omega_{2m}(F,Q_W)$, 
    we conclude that 
    the orbit of $v_2$ under the action of $\mathcal{G}$ is the following set:
    \begin{equation*}
	\Orb_{\mathcal{G}}(v_2) = \left\{ \left( -Q_W(u) \mid  u \mid 1
	 \right)\ \bigg|\ u \in W  \right\}.
    \end{equation*}
    Since the elements of this orbit are in one-to-one correspondence with the
    cosets of $\Omega_{2m}(F,Q_W)$ in $\mathcal{G}$, 
    it is enough to show the primitive
    action on $\Orb_{\mathcal{G}}(v_2)$. 
    
    Consider the action of $\mathcal{G}$ on 
    $\Orb_{\mathcal{G}}(v_2)$. A general element in $\mathcal{G}$ 
    acts on the elements of $\Orb_{\mathcal{G}}(v_2)$ in the following way:
    \begin{multline*}
	( -Q_W(u) \mid  u \mid 1 ) \ \mapsto \ 
	( -Q_W(u) - uABv^\T - Q(v) \mid
	    uA + v \mid 1) = \\
	    = (-Q_W(uA + v) \mid uA + v \mid 1).
    \end{multline*}
    Note that $uABv^\T = f_W( uA, v )$. 
    We see that this action is isomorphic to the action on $W$ defined by
	$u \mapsto uA + v$,
    where $u,v \in W$. In case when $A$ is the identity matrix, this map
    is a translation. On the other hand, taking $v = 0$, we obtain the
    action of $\Omega_{2m}(F,Q_W)$. 
    Denote the group generated by the described action on $W$ as
    $\mathrm{A\Omega}_{2m}(F,Q_W)$.
    
    The action of $G=\mathrm{A\Omega}_{2m}(F,Q_W)$ on 
    $W$ is primitive if any $G$-congruence on the elements of
    $W$ 
    is trivial, i.e. it is either equality or the universal relation. 
    Now suppose $\sim$ is a non-trivial $G$-congruence. In particular,
    $\sim$ is not the equality relation, so we may assume that there are two different
    elements $v_1,v_2 \in W$ such that $v_1 \sim v_2$. Translating both
    $v_1$ and $v_2$ by $-v_2$ and using the fact that $G$ preserves
    $\sim$, we obtain $v_1-v_2 \sim 0$. That is, $0 \sim v$ for some non-zero
    vector $v$. Denote by $O_{\lambda}$ the set
    \begin{equation*}
	O_{\lambda} = \left\{ u\in W\ \mid\ u\neq 0,\ Q_W(u) = \lambda \right\}
    \end{equation*}
    and let $\lambda = Q_W(v)$. Note that the group $\Omega_{2m}(F,Q_W)$
    acts transitively on $O_{\lambda}$ for all $\lambda \in F$, 
    so from $0 \sim v$ we obtain
    $0 \sim vA$ for all $A \in \Omega_{2m}(F,Q_W)$, and hence
    $0 \sim v_{\lambda}$ for all $v_{\lambda} \in O_{\lambda}$. 
    
    Suppose $v \in W$ is such
     that 
    $Q_W(v) = \lambda \neq 0$. Since $Q_W$ is of Witt index at least $2$,
    there exist two isotropic vectors $e_1,f_1 \in W$ spanning a totally isotropic
     subspace
    of dimension $2$. Let $\lambda_1 = f_W(v,e_1)$, $\lambda_2 = f_W(v,f_1)$ and
    denote \mbox{$u = \lambda_2 e_1 - \lambda_1 f_1$}. Then $Q_W(u) = 0$ and so 
    \begin{equation*}
	Q_W(v+u) = Q_W(v) + f_W(v,u) = Q(v) + \lambda_2 f_W(v,e_1) - \lambda_1 f_W(v,f_1)
	 = Q_W(v),
    \end{equation*}
    in other words, $f_W(v+u, v) = 0$. We readily see that $v+u \in O_{\lambda}$, and
    so $v \sim v+u$. Employing the translation by $-v$ we obtain $0 \sim u$. That is,
    $0$ is congruent to some non-zero isotropic vector. Of course, this means that 
    $0 \sim u_0$ for any $u_0 \in O_0$, since $\Omega_{2m}(F,Q_W)$ acts
    transitively on the set of isotropic vectors. 
    
    We say that two vectors $u,v \in W$ are in $\lambda$-relation 
    if $u \sim v$ and $Q(u-v) = \lambda$. From the previous discussion we know that 
    $0$ is in $0$-relation with some isotropic non-zero vector $u$. 
    Choose $w \in W$ such
    that $(u,w)$ is a hyperbolic pair. Then for all non-zero $\xi \in F$
    there is a chain $u \sim 0 \sim \xi w$. Indeed, since the action of 
    $\Omega_{2m}(F,Q_W)$ is transitive on the elements of $O_0$, then
    $0 \sim u_0$ for all $u_0 \in O_0$ and so by transitivity 
    $u_0 \sim u_0'$ for any two $u_0, u_0' \in O_0$. For any non-zero $\xi \in F$,
    $\xi w$ is a non-zero isotropic vector, hence the result. Since $u \sim \xi w$, 
    we can translate both sides by $-\xi w$ to get $u-\xi w \sim 0$, where
    $Q_W(u-\xi w) = Q_W(u) + \xi^2 Q_W(w) - \xi f_W(u, w) = -\xi$. Therefore, having
    the $0$-relation it is possible to obtain a $\xi$-relation for any $\xi \neq 0$
    in $F$. Finally, using the transitivity of $\Omega_{2m}(F,Q_W)$
    on every $O_{\xi}$, we obtain that $\sim$ is the universal congruence, and so
    the action of $\mathrm{A\Omega}_{2m}(F,Q_W)$ on $W$ is primitive.
\end{proof}

As we already know from Lemma \ref{lemma:1_stabiliser_omega}, the stabiliser in $\Omega_{2m+2}(F,Q)$
of an isotropic vector \mbox{$v_1\in V$} is a subgroup of shape $W\cn\Omega(F,Q_W)$.
We also note that every subgroup of $\Omega_{2m+2}(F,Q)$ containing 
$W\cn\Omega_{2m}(F,Q_W)$ as a subgroup, stabilises the \mbox{$1$-space} spanned by 
$v_1$.

\begin{theorem}
	\label{theorem:1_space_stab}
	Let $Q_W$ be of Witt index at least $2$. Any subgroup $H$ such that
	\begin{equation}
		W\cn\Omega_{2m}(F,Q_W) \leqslant H < \Omega_{2m+2}(F,Q),
	\end{equation}
	stabilises the $1$-space $\langle v_1 \rangle$. 
\end{theorem}

\section{Octonions}

In this section we discuss the definition and some properties of octonion algebras. 
Our usual point of reference will be \cite{SprVeld}. 

\begin{definition}
	Let $F$ be any field. An octonion algebra $\OO = \OO_F$ is an $8$-dimensional composition
	algebra, i.e. it is unital and admits a norm $\NN : \OO \rightarrow F$
	that is a quadratic form
	such that the polar form of\/ $\NN$ is non-degenerate and $\NN(xy) = \NN(x) \NN(y)$ for 
	all $x,y \in \OO$.  
\end{definition}

The multiplicative identity in $\OO$ is denoted as $1_{\OO}$, and throughout the paper we sometimes omit the subscript. Denote the polar form of $\NN$ by $\langle \cdot , \cdot \rangle$ and define the 
\textit{trace} of an octonion via
\begin{equation}
	\Tr(x) = \langle x, 1_{\OO} \rangle,
\end{equation}
where $1_{\OO}$ is the multiplicative identity in $\OO$. Octonion algebras are quadratic: an 
arbitrary element $x \in \OO$ satisfies the equation
\begin{equation}
	x^2 - \Tr(x)\cdot x + \NN(x)\cdot 1_{\OO} = 0.
\end{equation} 
\textit{Conjugation} in $\OO$ is the
mapping $\ovphantom: \OO \rightarrow \OO$ defined by
\begin{equation}
	\ovx = \Tr(x)\cdot 1_{\OO} - x.
\end{equation}
We call $\ovx$ the \textit{conjugate} of $x$. The following lemma summarises the properties of 
$\OO$ related to octonion conjugation. 

\begin{lemma}
	For all $x,y \in \OO$ the following identities hold:
	\begin{enumerate}[(i)]
		\item $x \ovx = \ovx x = \NN(x)\cdot 1_{\OO}$,
		\item $\bar{xy} = \ovy \ovx$, 
		\item $\bar{\ovx} = x$, 
		\item $\bar{x+y} = \ovx + \ovy$,
		\item $\NN(x) = \NN(\ovx)$,
		\item $\langle x,y \rangle = \langle \ovx, \ovy \rangle$. 
	\end{enumerate}
\end{lemma}

Furthermore, we have the following important properties. 

\begin{lemma}
	For all $x,y,z \in \OO$ the following identities hold: 
	\begin{enumerate}[(i)]
		\item $x(\ovx y) = \NN(x) y$,
		\item $(x\ovy) y = \NN(y) x$,
		\item $x (\ovy z) + y (\ovx z) = \langle x,y \rangle \cdot z$,
		\item $(x \ovy) z + (x \ovz) y = \langle y,z \rangle \cdot x$. 
	\end{enumerate}	 
\end{lemma}

Octonion algebras are alternative, which means that the octonion multiplication
satisfies the following laws.

\begin{lemma}
	For all $x,y \in \OO$ the following are true:
	\begin{enumerate}[(i)]
		\item $(xx)y = x(xy)$, 
		\item $(yx)x = y(xx)$,
		\item $(xy)x = x(yx)$. 
	\end{enumerate}
\end{lemma}

Finally, we must emphasise that octonionic multiplication is not 
associative and not commutative.
The following is one part of Theorem 1.6.2 in \cite{SprVeld}. 

\begin{theorem}
	An(y) $8$-dimensional composition algebra is neither associative nor commutative.
\end{theorem}

While the last theorem suggests that the calculations involving the elements of $\OO$ can be 
quite tedious and uncomfortable, in some cases the following lemmata can make our life easier. 

\begin{lemma}
	If $x,y,z \in \OO$, then $\Tr(xy) = \Tr(yx)$ and $\Tr(x(yz)) = \Tr((xy) z)$. 
\end{lemma}

Note that although we have $3$-associativity for the trace in general, we cannot
derive generalised associativity in this case. Although there is no associativity in
 general, in some cases the so-called
\textit{Moufang laws} can help us with bracketing.

\begin{lemma}
	For all $x,y,z\in \OO$, the following identities hold:
	\begin{equation}
		\begin{array}{r@{\;}c@{\;}l}
			x(yz)x & = & (xy)(zx), \\
			x(yzy) & = & ((xy)z)y, \\
			(xyx)z & = & x(y(xz)).
		\end{array}
	\end{equation}
\end{lemma}

An octonion $x \in \OO$ is \textit{invertible} if and only if $\NN(x) \neq 0$, in which case
\begin{equation}
x^{-1} = \NN(x)^{-1} \ovx.
\end{equation}
 If there exists an isotropic octonion in $\OO$, i.e. a non-zero
element $x \in \OO$ such that $\NN(x) = 0$, then we call $\OO$ 
a \textit{split octonion algebra}. One part of Theorem 1.8.1 in \cite{SprVeld} is the following result.

\begin{theorem}
	\label{theorem:1_unique_split_algebra}
	Over any given field $F$ there is a unique, up to isomorphism, split composition algebra.
\end{theorem}

It turns out that any isotropic octonion in the split octonion algebra 
$\OO$ over $F$ left-annihilates a $4$-dimensional
subspace and right-annihilates a $4$-dimensional subspace. 

\begin{lemma}
	\label{lemma:1_octonion_annihilator}
	Let $\OO$ be a split octonion algebra. Then for any isotropic $x \in \OO$, the 
	following is true:
	\begin{equation}
		\dim_F(\OO x) = \dim_F(x\OO) = 4. 
	\end{equation}
	Moreover, $\OO x$ is the set of octonions that are right-annihilated by $\ovx$, and
	$x \OO$ is the set of octonions that are left-annihilated by $\ovx$.
\end{lemma}

\begin{proof}
	We prove the statement for right multiplication by $x$. The proof
	for left multiplication is essentially the same. The map
	\begin{equation*}
		\begin{array}{r@{\;\;}c@{\;\;}l}
			R_x : \OO & \rightarrow & \OO \\
			y & \mapsto & yx
		\end{array}
	\end{equation*}
	is an $F$-linear map with $\im(R_x) = \OO x$, which is a totally isotropic
	subspace of $\OO$. Indeed, $\NN(yx) = \NN(y)\NN(x) = 0$ for any
	$y \in \OO$. Since the polar form of $\NN$ is non-degenerate,
	we conclude that $\dim_F(\OO x) \leqslant 4$.
	
	If $x \neq 0$ and $yx = 0$, then $y$ is isotropic for if that were not 
	the case, we would get $x = y^{-1}(yx) = y^{-1}\cdot 0 = 0$,
	a contradiction. It follows that \mbox{$\dim_F(\ker(R_x)) \leqslant 4$}.
	The Rank--Nullity theorem implies that \mbox{$\dim_F(\OO x) = \dim_F(\ker(R_x)) = 4$}. 
\end{proof}

Finally, we describe the centre of $\OO$ and also the elements in an octonion algebra that
`associate' with all other elements. 
By the centre of an octonion algebra $\OO$ we understand 
\mbox{$\{\, c \in \OO\ \mid\ cx = xc\ \mbox{for all}\ x \in \OO\, \}$}. In the literature, 
for example, in \cite{Schafer}, it is sometimes required that central 
elements also ``associate'' with all other elements. We do not require this in our definition,
however, it will become obvious that we have this property free of charge. 

\begin{lemma}
	\label{lemma:1_octonion_centre}
	The centre of an octonion algebra $\OO$ over $F$ is $F\cdot 1_{\OO}$.  
\end{lemma} 

\begin{lemma}
	Let $\OO$ be an octonion algebra over $F$. If $u \in \OO$ satisfies 
	\begin{equation}
		(xu)y = x(uy)
	\end{equation}
	for all $x,y \in \OO$, then $u \in F\cdot 1_{\OO}$. 
\end{lemma}

\begin{corollary}
    \label{cor:1_auub}
    Suppose that $u \in \OO$ is an invertible octonion. Then
    \begin{equation}
	(A \ovu   )(uB) = \NN(u) AB
    \end{equation}
    holds for all $A,B \in \OO$ if and only if
    $u \in F\cdot 1_{\OO}$.
\end{corollary}
   
\begin{remark}
	The statement of this corollary holds even if $u$ is not invertible, but this requires
	a different, more hands-on proof. 
\end{remark}

In the subsequent constructions we consider certain subalgebras of $\OO$. We say
that an $F$-subalgebra $S$ of $\OO$ is \textit{sociable} if $S$ contains $F\cdot 1_{\OO}$ and
 for all $x,y \in S$ and for all $z \in \OO$ we have $(xy)z = x(yz)$. 
 
Finally, we prove two technical lemmas which will be useful later in the paper.  First, we show that 
no invertible $F$-linear maps on $\OO$ can change the order of the octonion product. 

\begin{lemma}
    \label{lemma:1_white_phipsi1}
    There are no invertible $F$-linear maps $\phi,\psi : \OO \rightarrow \OO$ such that for all
    $A,B \in \OO$ it is true that $AB = (B\psi) (A\phi)$.
\end{lemma}

\begin{proof}
    For the sake of finding a contradiction, suppose that $\phi,\psi : \OO \rightarrow \OO$ are
    invertible $F$-linear maps such that the identity $AB = (B\psi)(A\phi)$ holds for all \mbox{$A,B\in \OO$}.
    In particular, substituting $A = 1_{\OO}$, we get $B = (B\psi) u$ for all $B \in \OO$, where $u = 1 \phi$,
    so $B \psi = B u^{-1}$ for all $B \in \OO$, which means that the map $\psi$ is right
    multiplication by $u^{-1}$. Note that the existence of $u^{-1}$ follows from the
    invertibility of the map $\psi$. Thus, our identity has the form $AB = (B u^{-1}) (A\phi)$
    for all $A,B \in \OO$. We can substitute $B = u$ which immediately gives us $A\phi = Au$ for
    all $A \in \OO$, so the map $\phi$ is right multiplication by $u$. Finally, we get
    $AB = (B u^{-1}) (A u)$ for all $A,B \in \OO$ and specifically for $B = 1_{\OO}$ we get
    $A = u^{-1} (A u)$, or likewise $u A = A u$ for all $A \in \OO$. Therefore $u$ is a
    scalar multiple of $1_{\OO}$, i.e. $u = \mu \cdot 1_{\OO}$ for some $\mu \in F$.
    Since the linear maps $\phi$ and $\psi$ are invertible, $\mu$ is non-zero, and we get
    $AB = (Bu^{-1})(Au) = (\mu^{-1} \mu \cdot 1_{\OO}) BA = BA$ for all $A,B \in \OO$, which is definitely
    not true as $\OO$ is not commutative.
\end{proof}

Second, we show that if two invertible linear maps commute with the octonion product, then
these are mutually invertible scalar multiplication maps.

\begin{lemma}
    \label{lemma:1_white_phipsi2}
    Suppose $\phi,\psi: \OO \rightarrow \OO$ are two invertible $F$-linear maps such that
    \mbox{$AB = (A\phi) (B\psi)$} for all $A,B \in \OO$. Then $\psi : x \mapsto \mu x$
    for some non-zero $\mu \in F$ and $\phi = \psi^{-1}$, i.e. $\phi : x \mapsto \mu^{-1} x$.
\end{lemma}

\begin{proof}
    Suppose $\phi, \psi: \OO \rightarrow \OO$ are $F$-linear maps such that
    $AB = (A\phi) (B\psi)$ for all $A,B \in \OO$. When $A=1_{\OO}$ we get $B\psi = uB$ for all
    $B \in \OO$ where $u = (1_{\OO}\phi)^{-1}$, so
    the map $\psi$ is left multiplication by $u$.
    Substituting $B = 1_{\OO}$ on the other hand gives
    us $A = (A\phi) (1_{\OO}\psi)$ for all $A$ and so $A\phi = A v$ where $v =
    (1_{\OO}\psi)^{-1}$, so $\phi$ is the right multiplication by $v$.
    Therefore the condition in this case becomes $AB = (Av)(uB)$
    for all $A,B \in \OO$. Substituting $B = u^{-1}$, we get $Au^{-1} = Av$ for all
    $A \in \OO$, and therefore $v = u^{-1}$, and our identity turns out to be
    $AB = (Au^{-1})(uB)$ for
    all $A,B\in \OO$. Now since $u$ is invertible, we can write $u^{-1} = N(u)^{-1} \ovu$.
    Finally, by Corollary \ref{cor:1_auub}, $u$ must be a scalar multiple of $1_{\OO}$,
    i.e. $u = \mu \cdot 1_{\OO}$.
\end{proof}

The statements in Lemmas \ref{lemma:1_white_phipsi1} and \ref{lemma:1_white_phipsi2} are
true even when $\OO$ is not split.
 
\section{A basis for the split octonion algebra} 

Theorem \ref{theorem:1_unique_split_algebra} 
allows us to choose a basis for $\OO$ and use it in our further constructions.
Otherwise speaking, we redefine the split octonion algebra $\OO$ over $F$ in the following way. 

\begin{definition}
	If $F$ is any field, then the split octonion algebra over $F$ is defined as an
	$8$-dimensional vector space $\OO = \OO_F$ with basis $\{e_i \mid i \in \pm I\}$,
	where $I = \{0,1,\omega,\ombb\}$, $\pm I = \{\pm 0, \pm 1, \pm \omega,
	\pm \ombb\}$ and bilinear multiplication given by the following table.
\end{definition}

    \begin{center}
        \begin{tabular}{ c || c | c | c | c | c | c | c | c | }
 & \xone & \xtwo & \xthree & \xfour & \xfive & \xsix & \xseven & \xeight \\ \hline \hline
 \xone & $0$ & $0$ & $0$ & $0$ & \xone & \xtwo & \m\xthree & \m\xfour \\ \hline
 \xtwo & $0$ & $0$ & \m\xone & \xtwo & $0$ & $0$ & \m\xfive & \xsix \\ \hline
 \xthree & $0$ & \xone & $0$ & \xthree & $0$ & \m\xfive & $0$ & \m\xseven \\ \hline
 \xfour & \xone & $0$ & $0$ & \xfour & $0$ & \xsix & \xseven & $0$\\ \hline
 \xfive & $0$ & \xtwo & \xthree & $0$ & \xfive & $0$ & $0$ & \xeight\\ \hline
 \xsix & \m\xtwo & $0$ & \m\xfour & $0$ & \xsix & $0$ & \xeight & $0$ \\ \hline
 \xseven & \xthree & \m\xfour & $0$ & $0$ & \xseven & \m\xeight & $0$ & $0$ \\ \hline
 \xeight & \m\xfive & \m\xsix & \xseven & \xeight & $0$ & $0$ & $0$ & $0$\\ \hline
        \end{tabular}
    \end{center}
In other words, we get
    \begin{enumerate}[(i)]
    	\item $e_1 e_{\omega} = -e_{\omega} e_1 = e_{-\ombb}$;
    	\item $e_1 e_0 = -e_{-0} e_1 = e_1$; 
    	\item $e_{-1} e_1 = - e_0$ and $e_0 e_0 = e_0$;
    \end{enumerate}
and images under negating all subscripts (including $0$), and multiplying all subscripts
by $\omega$, where $\omega^2 = \ombb$ and $\omega \ombb = 1$. All other
products of basis vectors are $0$. Essentially, this is the same basis as given
in section 4.3.4 of \cite{FSG}. Thus, $e_0$ and $e_{-0}$ are orthogonal idempotents
with $e_0 + e_{-0} = 1_{\OO}$. Now, if $x = \sum_{i \in \pm I} \lambda_i e_i$, 
then the trace and the norm can be defined in the following way:
\begin{equation}
	\begin{array}{r@{\;}c@{\;}l}
		\Tr(x) & = & \lambda_0 + \lambda_{-0}, \\
		\NN(x) & = & \lambda_{-1} \lambda_1 + \lambda_{\omg} \lambda_{-\omg} + 
			\lambda_{\omega}\lambda_{-\omega} + \lambda_0 \lambda_{-0}.
	\end{array}
\end{equation}
Note that whenever we obtain an octonion which is a scalar multiple of $1_{\OO}$, we
understand it as a field element.
The fact that our newly defined algebra $\OO$ is indeed a composition algebra can be verified 
by a tedious but straightforward computation. Note that $\NN(e_i) = 0$ for $i \neq \pm 0$,
so $\OO$ is indeed a split octonion algebra.

The involution $x \mapsto \ovx$ is the extension by linearity of 
\begin{equation}
	e_i \mapsto - e_i\ (i \neq \pm 0),\ e_0 \leftrightarrow e_{-0}.
\end{equation}

\section{Albert vectors}

For further discussion we consider $\OO = \OO_F$ to be an octonion algebra 
over the field $F$. The Albert space $\J=\J_F$
is the $27$-dimensional vector space spanned by the elements of the form
\begin{equation}
	(a,b,c\mid A,B,C) = \begin{bmatrix}
		a & C & \ovB \\
		\ovC & b & A \\
		B & \ovA & c
	\end{bmatrix},
\end{equation}
where $a,b,c,A,B,C \in \OO$ and furthermore $a,b,c \in \langle 1_{\OO} \rangle$. To denote
certain subspaces of $\J$ we use the following intuitive notation. The $10$-dimensional
subspace spanned by the Albert vectors of the form $(a,b,0\mid 0,0,C)$ is denoted
$\J_{10}^{abC}$, while the $8$-space spanned by the vectors $(0,0,0\mid A,0,0)$ 
is denoted $\J_{8}^{A}$ and so on. That is, the subscript determines the dimension
and the superscript shows which of the six `co{\"o}rdinates' we use to span the
corresponding subspace. Of course, this notation is by no means complete as it does not
allow us to denote any possible subspace of $\J$. If this is the case, we specify 
the spanning vectors and denote the corresponding space in some other manner.

Lacking the associativity in $\OO$ we also need to be slightly careful when we calculate
the determinant of $X$. For these purposes we define the Dickson--Freudenthal determinant
as
\begin{equation}
	\fdet(X) = abc - aA\ovA - bB\ovB - cC\ovC + \Tr(ABC).
\end{equation}
This is a cubic form on $\J$ and it can be shown that it is equivalent to the 
original Dickson's cubic form used to construct the group of type $\E_6$. 

We define the group $\SE_6(F)$ or $\SE_6(F,\OO)$ if we want to specify the octonion 
algebra, to be the group of all $F$-linear maps on $\J$ preserving
 the Dickson--Freudenthal determinant. If $F = \Fq$, then we denote this by $\SE_6(q)$.
  The group $\EE_6(F)$ is defined as the quotient
 of $\SE_6(F)$ by its centre.
Suppose $M$ is a $3\times 3$ matrix written over $\OO$. If $M$ is written over any 
sociable subalgebra of $\OO$, then for an element $X \in \J$ the mapping 
$X \mapsto \ovM^{\T} X M$ makes sense. Indeed, every entry in the matrix $\ovM^{\T} X M$
is a sum of the terms of the form $m_1 x m_2$, where $m_1$ and $m_2$ belong to the same
sociable subalgebra, and so $(m_1 x) m_2 = m_1 (x m_2)$.

Suppose $X = (a,b,c \mid A,B,C)$ and $Y = (d,e,f\mid D,E,F)$. Define the mixed form
$M(Y,X)$ as
\begin{multline}
	M(Y,X) = bcd + ace + abf - d A\ovA  -e B \ovB - f C\ovC \\
	- a(D\ovA + A\ovD) - b(E\ovB + B\ovE) - c(F\ovC + C\ovF) \\
	+ \Tr(DBC + ECA + FAB). 
\end{multline}
Note that if $F \neq \mathbb{F}_2$, then $M(X,Y)$ can be obtained from the 
Dickson--Freudenthal
determinant, for we have
\begin{multline}
	M(X,Y) = \frac{1}{\alpha (\alpha - 1)} \fdet(X + \alpha Y) \\
	- \frac{1}{\alpha - 1} \fdet(X+Y) + \frac{1}{\alpha} \fdet(X) - (\alpha+1)\fdet(Y),
\end{multline}
for any $\alpha \not\in \{0,1\}$.

We colour the non-zero Albert vectors in $\J$ according to the following rules. 

\begin{definition}
	A non-zero Albert vector $X \in \J$ is called
	\begin{enumerate}[(i)]
		\item \textit{white} if $M(Y,X) = 0$ for all $Y \in \J$;
		\item \textit{grey} if $\Delta(X) = 0$ and there exists $Y \in \J$ such that 
			$M(Y,X) \neq 0$;
		\item \textit{black} if $\Delta(X) \neq 0$ and $X$ is not white.  
	\end{enumerate}
	A \textit{white/grey/black} point is a $1$-dimensional subspace of $\J$ spanned by a 
	white/grey/\newline black vector. 
\end{definition}

For example, the vector $(0,0,1\mid 0,0,0)$ is white, because if $Y$
is an arbitrary Albert vector, then $M(Y,X) = 0$. Similarly, $(\lambda,1,1\mid 0,0,0)$,
where $\lambda \neq 0$, is black, since in this case $\Delta(X) = \lambda \neq 0$,
and it is certainly not white as there exists $Y = (a,b,c\mid A,B,C)$ such that 
$M(Y,X)\neq 0$:
\begin{equation}
	M(Y,X) = \lambda (bc - A\ovA) + (ac - B\ovB) + (ab - C\ovC). 
\end{equation}
Taking, for instance, $Y = (0,1,1\mid 0,0,0)$, we get $M(Y,X) = \lambda \neq 0$. Finally, 
$(0,1,1\mid 0,0,0)$ is grey as $\Delta(X) = 0$ and for $Y = (a,b,c\mid A,B,C)$ the
value of $M$ is given by
\begin{equation}
	M(Y,X) = (ac - B\ovB) + (ab - C\ovC),
\end{equation}
so we may take $Y = (1,1,0\mid 0,0,0)$ to get $M(Y,X) = 1 \neq 0$. 
Later we will also show that if $X$ is white, then $\fdet(X) = 0$. 
The terms white, grey and black were introduced by Cohen and Cooperstein 
\cite{CC}. In the paper by Aschbacher \cite{Asch1} they are called
'singular', 'brilliant non-singular' and 'dark' respectively. 
Jacobson \cite{Jac2} uses the terms 'rank $1$', 'rank $2$' and 'rank $3$'. 

It is clear that the action of $\SE_6(F)$ preserves the colour, except possibly in case
$F = \mathbb{F}_2$, when white and grey vectors might conceivably be intermixed. However, 
as we will see now, this does not happen. 

Let $X = (a,b,c\mid A,B,C)$ and consider the case $Y = (0,0,1\mid 0,0,0)$.  Then  
\mbox{$\Delta(X+Y) - \Delta(X) = ab - C\ovC$}, which is a quadratic form whose
radical in $\J$ is 
\mbox{$17$-dimensional}.  If $Y = (0,1,1\mid 0,0,0)$ and 
$F=\mathbb{F}_2$, then $a^2 = a$ and so the form  \mbox{$\Delta(X+Y) - \Delta(X) = a+ab+ac-B\ovB-C\ovC$} is 
quadratic with $9$-dimensional radical. This
shows that $(0,0,1\mid 0,0,0)$ and $(0,1,1\mid 0,0,0)$ are in different orbits 
of the isometry group for any field.

\section{Some elements of $\SE_6(F)$}

Throughout this section, let $X = (a,b,c\mid A,B,C)$ to be an arbitrary element of 
$\J = \J_F$.\ We encode some of the elements of $\SE_6(F)$ by the $3\times 3$
 matrices written over
social subalgebras of $\OO=\OO_{F}$. As we mentioned before, if such a matrix $M$ is written
over any sociable subalgebra of $\OO$, then the expression $\ovM^{\T} X M$ makes sense.
Furthermore, the action of the form $X \mapsto \ovM^{\T} X M$ is obviously $F$-linear.
If two matrices $M$ and $N$ are written over the same sociable subalgebra, then we have 
enough associativity to see that the action by the product $MN$ is the same as the product
of the actions, that is
\begin{equation}
	(\ovN \ovM)^{\T} X (MN) = \ovN^{\T} (\ovM^{\T} X M) N.
\end{equation} 
In general, the action by the product of two matrices is not defined whereas the product 
of the actions still is. Note that also $-\II_3$ acts trivially on $\J$. 

We first notice that the elements 
\begin{equation}
	\delta = \begin{bmatrix}
		0 & 1 & 0 \\
		1 & 0 & 0 \\
		0 & 0 & 1
	\end{bmatrix},\quad 
	\tau = \begin{bmatrix}
		0 & 1 & 0 \\
		0 & 0 & 1 \\
		1 & 0 & 0
	\end{bmatrix}
\end{equation}
preserve the Dickson--Freudenthal determinant. Their actions are given by
\begin{equation}
	\begin{array}{r@{\;\;}c@{\;\;}l}
		\delta: (a,b,c\mid A,B,C) & \mapsto & (b,a,c\mid \ovB, \ovA, \ovC), \\
		\tau: (a,b,c\mid A,B,C) & \mapsto & (c,a,b\mid C,A,B). 
	\end{array}
\end{equation}

Now let $x$ be any octonion and consider the matrices
\begin{equation}
	M_x = \begin{bmatrix}
		1 & x & 0 \\
		0 & 1 & 0 \\
		0 & 0 & 1 
	\end{bmatrix},\quad 
	M_x' = \begin{bmatrix}
		1 & 0 & 0 \\
		0 & 1 & x \\
		0 & 0 & 1
	\end{bmatrix},\quad 
	M_x'' = \begin{bmatrix}
		1 & 0 & 0 \\
		0 & 1 & 0 \\
		x & 0 & 1
	\end{bmatrix}.
\end{equation}
Note that the elements $M_x'$, $M_x''$ can be obtained from $M_x$ by applying 
the triality element $\tau$, so to show that all three families described above preserve
the Dickson--Freudenthal determinant, we only need to consider one of them. 

\begin{lemma}
	\label{lemma:1_Mx_det}
	The elements $M_x$, where $x \in \OO$ is any octonion, preserve the Dickson--Freudenthal
	determinant, and hence they encode the elements of $\SE_6(F)$. 
\end{lemma} 

\begin{proof}
	The action of $M_x$ on $\J$ is given by
	\begin{equation*}
		\begin{array}{r@{\;\;}c@{\;\;}l}
		M_x : (a,b,c\mid A,B,C) & \mapsto & (a,b+a\NN(x)+\Tr(\ovx C),c \mid 
										A+\ovx\ovB, B, C+ax).
		\end{array}
	\end{equation*}
	The individual terms in the Dickson--Freudenthal determinant are being mapped in
	the following way:
	\begin{equation*}
		\begin{array}{r@{\;\;}c@{\;\;}l}
			abc & \mapsto & abc + a^2c\NN(x) + ac\Tr(\ovx C), \\
			-aA\ovA & \mapsto & -aA\ovA - a\Tr(ABx) - a\NN(x)\NN(B), \\
			-bB\ovB & \mapsto & -bB\ovB - a\NN(x)\NN(B) - \Tr(\ovx C)B\ovB, \\
			-cC\ovC & \mapsto & -cC\ovC - ac\Tr(\ovx C) - a^2c\NN(x), \\
			\Tr(ABC) & \mapsto & \Tr(ABC) + B\ovB\Tr(\ovx C) + 2a\NN(x)\NN(B) + a\Tr(ABx). 
		\end{array}
	\end{equation*}
	It is visibly obvious now that all the necessary terms on the right-hand side cancel out,
	so the result follows.
\end{proof}

It is obvious enough that we can also consider
\begin{equation}
	L_x = \begin{bmatrix}
		1 & 0 & 0 \\
		x & 1 & 0 \\
		0 & 0 & 1
	\end{bmatrix},\quad 
	L_x' = \begin{bmatrix}
		1 & 0 & 0 \\
		0 & 1 & 0 \\
		0 & x & 1 
	\end{bmatrix},\quad 
	L_x'' = \begin{bmatrix}
		1 & 0 & x \\
		0 & 1 & 0 \\
		0 & 0 & 1
	\end{bmatrix}
\end{equation}
for an arbitrary $x \in \OO$. 
A similar straightforward calculation as in Lemma \ref{lemma:1_Mx_det} can be performed
to show that these are also the elements of $\SE_6(F)$. 
Further in this paper we will be able to show that the actions of the elements
$M_x$, $M_x',M_x'',L_x,L_x'$ and $L_x''$ generate the whole group $\SE_6(F)$.

Finally, we consider the elements of the form 
\begin{equation}
	P_u = \begin{bmatrix}
		u & 0 & 0 \\
		0 & \ovu & 0 \\
		0 & 0 & 1
	\end{bmatrix},\quad 
	P_u' = \begin{bmatrix}
		1 & 0 & 0 \\
		0 & u & 0 \\
		0 & 0 & \ovu
	\end{bmatrix},\quad 
	P_u'' = \begin{bmatrix}
		\ovu & 0 & 0 \\
		0 & 1 & 0 \\
		0 & 0 & u
	\end{bmatrix},
\end{equation}
where $u$ is an octonion of norm one. The action of the element $P_u$ on $\J$ is given by
\begin{equation}
    \begin{array}{r@{\;\;}c@{\;\;}l}
	P_u: (a,b,c\mid A, B, C) & \mapsto & (a,b,c\mid uA, Bu, \ovu  C \ovu ).
    \end{array}
\end{equation}
It is a matter of straightforward computation to show that the elements $P_u$ preserve
the Dickson--Freudenthal determinant. Indeed, we have
\begin{equation}
    \begin{array}{r@{\;\;}c@{\;\;}l}
	abc & \mapsto & abc, \\
	a A \ovA  & \mapsto & a(uA)(\ovA \ovu ) = aN(uA) = a\NN(A).\NN(u) = a A\ovA , \\
	b B \ovB  & \mapsto & b(B u)(\ovu  \ovB ) = b\NN(B u) = b B \ovB , \\
	c C \ovC  & \mapsto & c(\ovu  C \ovu )(u \ovC  u) =
	    c\NN(\ovu  C \ovu  ) = c \NN(C).\NN(u)^2 = c C \ovC ,
    \end{array}
\end{equation}
and for the last term we get
\begin{multline}
    \Tr((uA)(Bu)(\ovu  C\ovu  )) = \Tr((\ovu  C\ovu  )(uA)(Bu)) =
    \Tr((\ovu  (C(\ovu  (uA))))(Bu)) \\
    = \Tr((\ovu  (CA))(Bu)) = \Tr((Bu)(\ovu  (CA))) = \Tr(B(u(\ovu  (CA)))) = \Tr(B(CA))  \\
    = \Tr((BC)A) = \Tr(ABC).
\end{multline}
On the other hand, it is not difficult to see that $P_u = M_{u-1} \cdot L_1 \cdot M_{u^{-1}-1} \cdot L_{-u}$,
so the fact that the matrices $P_u$ preserve the determinant follows from the calculations already done for
the elements $M_x$ and $L_x$. We also notice that the elements $P_u$ preserve
the quadratic form $\QQ_8^C$ defined on $\J_8^C$ via
\begin{equation}
	\QQ_8^C( (0,0,0\mid 0,0,C) ) = C\ovC.
\end{equation}

We finish this section by showing that the action of the elements $P_u$ on $\J_{10}^{abC}$, as $u$ ranges
through all the octonions of norm one, is that of $\Omega_8(F,\QQ_8^C)$ when $\OO$ is split.

\begin{lemma}
	\label{lemma:1_spin8plus}
	If $\OO$ is split, the actions of the elements $P_u$ on 
	$\J_8^C$, as $u$ ranges through	all the octonions of norm one, 
	generate a group of type 
	$\Omega_8^+(F)$. The action on $\J_{10}^{abC}$ is also that 
	of $\Omega_8^+(F)$.
\end{lemma}

\begin{proof}
Consider the action on the last octonionic `co\"{o}rdinate', i.e.
$C \mapsto \ovu   C \ovu  $. We will show now that this map can be represented as a
product of two reflexions. To avoid any predicaments in characteristic $2$, we notice that
since $\langle x,y \rangle = \Tr(x\ovy )$, we get
\begin{equation}
\frac{2 \langle x, y \rangle}{\langle y,y \rangle} =
\frac{\langle x, y \rangle}{\NN(y)}.
\end{equation}
Now, the reflexion in the hyperplane orthogonal to an arbitrary element $v\in \OO$ is the map
\begin{equation}
    r_v : x \mapsto     x - \frac{\Tr(x \ovu  )}{\NN(u)} \cdot u =
	 x - \frac{ x\ovu  + u\ovx  }{\NN(u)} \cdot u =
    x - \frac{(x\ovu  )u - u\ovx u}{\NN(u)} = - \frac{u\ovx u}{\NN(u)},
\end{equation}
It is easy to see now that the given
action of $P_u$ on $\J_8^C$ is the composition $r_u \circ r_1$. As $u$ ranges through all 
octonions of norm one, we get the action of $\Omega_8(F,\QQ_8^C)$ on $\J_8^C$. Since we
assume that $\OO$ is split, the form $\QQ_8$ is of plus type, so we may denote this
group as $\Omega_8^+(F)$. When acting on $\J_{10}^{abC}$, the space $\J_2^{ab}$ is fixed
pointwise and the form $ab-C\ovC$ is
preserved, so we again get the action of $\Omega_8^+(F)$. 

\end{proof}

\section{Action of $\SE_6(F)$ on white points}

In this paper we will be mostly interested in the action of $\SE_6(F)$ on the 
white points. Let $X = (a,b,c\mid A,B,C)$ be an arbitrary white vector. 
A white vector $W$ determines the quadratic form $\fdet(X+W) - \fdet(X) = M(W,X)$ on $\J$. 
Its radical is $17$-dimensional and for any non-zero $\lambda \in F$ it coincides with the radical 
of the form determined by $\lambda W$.  Indeed,  \mbox{$\fdet(X+\lambda W) - \fdet(X) = \lambda( \fdet(X+W) - \fdet(X) )$}. 
Thus, the $17$-dimensional
space is determined by the white point $\langle W \rangle$. 

For example, for the white vector $(0,0,1\mid 0,0,0)$ the quadratic form is 
$ab - C\ovC$, whose radical is $\J_{17}^{cAB}$. For the vector
$(0,0,0\mid 0,0,D)$ with $D \neq 0 = D\ovD$ the form is 
$\hat{Q}(X) = \Tr( D (AB - c\ovC) )$ with
\mbox{$\hat{B}(X,Y) = \Tr( D ( AB' + A'B - c\ovC' - c'\ovC) )$} being its polar form,
where
$Y = (a',b',c'\mid A',B',C')$. Now $X$ is in the radical of $\hat{Q}$ if 
and only if $\hat{Q}(X) = 0$ and $\hat{B}(X,Y) = 0$ for all $Y$. 
Taking \mbox{$Y = (a',b',1\mid 0,0,0)$} gives us $\Tr(D\ovC) = 0$ and taking 
$Y = (a',b',0\mid 0,B',0)$ gives us \mbox{$\Tr(DAB' ) = \Tr((DA)B' ) = 0$}
for all $B'$, so $DA = 0$. If $Y = (a',b',0\mid A',0,0)$ then \mbox{$\Tr(D(A'B)) = \Tr((BD)A') = 0$} for all $A'$, so we get $BD = 0$. Finally, setting
$Y = (a',b',0\mid 0,0,C')$ gives us $\Tr(cD\ovC') = 0$ for all $\ovC'$, so
$cD = 0$, and thus $c = 0$. Therefore the radical is
\begin{equation}
	\left\{ (a,b,0 \mid A,B,C)\ \big|\ DA = BD = \Tr(D\ovC) = 0 \right\}.
\end{equation}
To obtain $17$-spaces determined by other ``co\"{o}rdinate'' white vectors we apply 
a suitable power of $\tau$ to these two. 

Next, we derive 
a system of conditions for an arbitrary vector $X \in \J$ to be white. 

\begin{lemma}
	An Albert vector $X = (a,b,c\mid A,B,C)$ is white if and only if the following 
	conditions hold:
	\begin{equation}
		\left.\begin{array}{r@{\;}c@{\;}l}
			A \ovA & = & bc, \\
			B \ovB & = & ca, \\
			C \ovC & = & ab, \\
			AB & = & c\ovC, \\
			BC & = & a\ovA, \\
			CA & = & b\ovB.
		\end{array}
		\right\}.
	\end{equation}
	If $X$ is white, then $\fdet(X) = 0$. 
\end{lemma}

\begin{proof}
	Let $Y = (d,e,f\mid D,E,F)$. We rewrite $M(Y,X)$ in the form
	\begin{multline*}
		M(Y,X) = (bc - A \ovA)d + (ac - B \ovB) e + (ab - C \ovC) f \\
			+ \Tr(D(BC-a\ovA) + Q(CA - b\ovB) + R(AB - c\ovC)).
	\end{multline*}
It is visibly clear now that if all the conditions in the statement are satisfied,
then $M(Y,X) = 0$. Now, taking $Y = (1,0,0\mid 0,0,0)$ forces
$bc-A\ovA = 0$. Similarly, we may take $Y = (0,1,0\mid 0,0,0)$ to get 
$ac - B\ovB = 0$ and, say, $Y = (0,0,0\mid D,0,0)$ to obtain
$\Tr(D(BC - a\ovA)) = 0$ which forces $BC - a\ovA = 0$ as $D \in \OO$ can
be arbitrary. The other conditions are proved similarly.

Finally, if $X$ is white, then we get $\Tr(ABC) = \Tr(a A\ovA) = \Tr(abc) = 2abc$. 
Also $b B\ovB = bca$, and so on. Overall we get 
\begin{equation*}
	\fdet(X) = abc - abc - bca - cab + 2abc = 0
\end{equation*}
as required. This completes the proof.
\end{proof}

Finally, we investigate the orbits of $\SE_6(F)$ on Albert vectors. One of our main
goals is to show that $\SE_6(F)$ acts transitively on white points. 

\begin{lemma}
	Suppose $X$ is an arbitrary Albert vector. Then $X$ can be mapped under the action
	of $\SE_6(F)$ to a vector of the form $(a,b,c\mid 0,0,0)$ with $(a,b,c) \neq 
	(0,0,0)$. In case when $\OO$ is split, $X$ can be mapped to precisely one of the
	 following:
	\begin{enumerate}[(i)]
		\item $(0,0,1\mid 0,0,0)$, a white vector;
		\item $(0,1,1\mid 0,0,0)$, a grey vector; or
		\item $(\lambda,1,1\mid 0,0,0)$ where $\lambda \neq 0$, a black vector.
	\end{enumerate}
	In the last case there is one orbit for each non-zero value of $\lambda$. 
\end{lemma}

\begin{proof}
	These vectors are indeed in the different orbits, except possibly for the white 
	and grey vectors, since they have different values of $\fdet$. We have already
	shown that these particular 
	white and grey vectors are in different orbits in case of any field. 
	
	First, we show that each orbit of $\SE_6(F)$ contains an Albert vector of the form
	$(a,b,c\mid 0,0,0)$. Suppose that $X = (a,b,c\mid A,B,C)$ is non-zero. If
	\mbox{$(a,b,c) = (0,0,0)$}, then after applying the triality element $\tau$ a suitable number
	of times we may assume $C \neq 0$. Consider the action of the element $L_x$ on
	the Albert vector $(0,0,0\mid A,B,C)$: 
	\begin{equation*}
		\begin{array}{r@{\;\;}c@{\;\;}l}
			L_x: (0,0,0\mid A,B,C) & \mapsto & (\Tr(Cx),0,0\mid A,B+\ovA x, C),
		\end{array}
	\end{equation*}
	so we are allowed to choose orbit representatives with $(a,b,c) \neq (0,0,0)$. 
	
	As before, using a suitable power of $\tau$, we may assume $c \neq 0$. Now we apply 
	the element $M_x$ with $x = -c^{-1}B$ to $X$, which gives the vector
	$(a,b,c\mid A,0,C)$, where the `co\"{o}rdinate' $c$ stays the same, while 
	$a,b,A,C$ are possibly different. Next,  $(a,b,c\mid A,0,C)$ is being mapped
	to $(a,b,c\mid 0,0,C)$ under the action of $L_x$ with
	$x = -c^{-1}A$, where the value of $c$ stays the same while the values of 
	$a,b,C$ may be adjusted. 
	
	If $a=b=0$, $C \neq 0$, then we apply the element $L_x$ with $x$ such that 
	\mbox{$\Tr(Cx)\neq 0$} to get $(\Tr(Cx),0,c\mid 0,0,C)$, i.e.
	we may assume that \mbox{$a \neq 0$}. With the latter assumption we apply the element
	$M_x$ with $x = -a^{-1} C$ to \mbox{$(a,b,c\mid 0,0,C)$} to get a vector of the form
	$(a,b,c\mid 0,0,0)$ with the value of $b$ being adjusted. 
	
	Finally, we use the elements $\tau$, $P_u$ and $P_v''$ to standardise $(a,b,c\mid 0,0,0)$ to one the forms in the statement. 
\end{proof}

Note that the last part of the proof of this lemma used the fact that the map
$\NN : \OO \rightarrow F$ is onto, which is the case when $\OO$ is split. 
However, this is not true in any octonion algebra, which possibly leads 
to a bigger number of orbits. A vector of the form $(a,b,c\mid 0,0,0)$ is 
white if and only if precisely one of the $a,b,c$ is non-zero, so we get 
the transitive action of $\SE_6(F)$ on white points regardless of the 
chosen octonion algebra. 

Furthermore, we used the fact that $\NN$ is a non-singular
quadratic form on $\OO$, i.e. provided $C \neq 0$, the map $x \mapsto \Tr(Cx)$ is 
surjective. This follows from the fact that the norm is a non-singular quadratic form on
$\OO$, something that should hold for any octonion algebra.

Later we will use the transitivity on white points to calculate the group order
in case $F = \Fq$ by finding the stabiliser of a white point and calculating the
number of white points in case of a finite field. 

\begin{lemma}
	Let $\OO$ be an arbitrary octonion algebra over $F$.
	Let $X \in \J$ be white and let $\J_{17}$ be the 
	$17$-dimensional subspace of $\J$ determined by $X$. The stabiliser in
	$\SE_6(F)$ of $\langle X \rangle$, and even of $X$, is transitive on the white points 
	spanned by the vectors in $\J_{17} \setminus \langle X \rangle$ (there are no such
	white points when $\OO$ is non-split). It is also
	transitive on the white points spanned by the vectors in $\J \setminus
	\J_{17}$. 
\end{lemma}

\begin{proof}
	Without loss of generality assume $X = (0,0,1\mid 0,0,0)$. As we know, the white 
	point $\langle X \rangle$ determines the $17$-space $\J_{17}^{cAB}$. We also note that
	$X$ is stabilised by the actions of the elements $M_x$, $L_x$, $M_x'$ and $L_x''$. Those
	act on the elements in $\J_{17}^{cAB}$ in the following way:
	\begin{equation*}
		\begin{array}{r@{\;\;}c@{\;\;}l}
			M_x : (0,0,c\mid A,B,0) & \mapsto & (0,0,c\mid A+\ovx \ovB,B,0), \\
			L_x : (0,0,c\mid A,B,0) & \mapsto & (0,0,c\mid A,B+\ovA x,0), \\
			M_x' : (0,0,c\mid A,B,0) & \mapsto & (0,0,c+\Tr(\ovx A) \mid A,B,0), \\
			L_x'' : (0,0,c\mid A,B,0) & \mapsto & (0,0,c+\Tr(B x) \mid A,B,0).
		\end{array}
	\end{equation*}
	It follows that a general white vector $(0,0,c\mid A,B,0) \in \J_{17}^{cAB}\setminus
	 \langle X
	\rangle$ can easily be mapped to $(0,0,0\mid A,B,0)$ using the action of $M_x'$ or 
	$L_x''$ for some suitable $x \in \OO$. A vector $(0,0,0\mid A,B,0)$ is white if
	$(A,B) \neq (0,0)$ and $A\ovA = B\ovB = AB = 0$. It is obvious enough that 
	$\J_{17}^{cAB} \setminus \langle X \rangle$ does not contain white vectors if $\OO$ is not split,
	so we only need to show transitivity on the corresponding white points in case
	when $\OO$ is split.
	
	If $B = 0$ then evidently $A \neq 0$ and so we can apply the duality element $\delta$
	to obtain a white vector of the form $(0,0,0\mid A,B,0)$ with $B \neq 0$. 
	If now $A \neq 0$, we act by $M_x$ to obtain $(0,0,0\mid A+\ovx \ovB,B,0)$. Our 
	aim is to show that there exists such $x \in \OO$ that $A+\ovx \ovB = 0$. 
	Denote $U = \left\{ y \in \OO\ \big|\ \ovy B = 0\right\}$. Since for all
	$x \in \OO$ we have $(\ovx \ovB) B = \ovx (\ovB B) = 0$, we conclude
	that $\OO \ovB \leqslant U$. Furthermore, we know that both
	subspaces are four-dimensional, so $\OO\ovB = U$. As $AB = 0$, we have $A \in U$, and
	therefore there exists $y = \ovx\ovB \in U$ such that $A+y = 0$.
	
	Now, the elements $P_u''$ with $\NN(u)=1$ act on the Albert vectors of the form 
	$(0,0,0\mid 0,B,0)$ as
	\begin{equation*}
		\begin{array}{r@{\;\;}c@{\;\;}l}
			(0,0,0\mid 0,B,0) & \mapsto & (0,0,0\mid 0,\ovu B \ovu,0),
		\end{array}
	\end{equation*}
	and as $u$ ranges through all the octonions of norm $1$ the action generated is that
	of $\Omega_8^+(F)$ which in case when $\OO$ is split is transitive on isotropic vectors,
	 i.e. those with $B\ovB = 0$.
	 It follows that $\SE_6(F)$ is indeed transitive on the white
	 points spanned by the vectors in $\J_{17}^{cAB} \setminus \langle X \rangle$.
	 
	 To show the transitivity on white points spanned by the vectors in 
	 $\J\setminus \J_{17}^{cAB}$ we prove that every white point 
	 spanned by a white vector $(a,b,c\mid A,B,C)\in \J\setminus \J_{17}^{cAB}$ 
	 can be mapped to the white point spanned by $(1,0,0\mid 0,0,0)$. Note that
	 we require $(a,b,C) \neq (0,0,0)$. 
	 
	 In case $(a,b) = (0,0)$ we choose $x\in \OO$ such that $\Tr(Cx) \neq 0$ and
	 apply the element $L_x$, which maps our vector $(0,0,c\mid A,B,C)$ to
	 $(\Tr(Cx),0,c\mid A,B+\ovA x,C)$. If, on the other hand, $a = 0$ and 
	 $b \neq 0$, we apply $\delta$. Hence, we may assume that we deal with a vector
	 $(a,b,c\mid A,B,C)$ with $a \neq 0$. Take $x = -a^{-1}C$ and act by the element
	 $M_x$: 
	 \begin{equation*}
	 	\begin{array}{r@{\;\;}c@{\;\;}l}
	 		M_x : (a,b,c\mid A,B,C) & \mapsto & 
	 				(a,b+ aa^{-2} C\ovC - \Tr(a^{-1}\ovC C),c \mid 
	 						A - a^{-1}\ovC \ovB, B, 0).
	 	\end{array}
	 \end{equation*}
	 The whiteness conditions imply $C\ovC = ab$ and $BC = a\ovA$, so additionally we get 
	 $b+ aa^{-2} C\ovC - \Tr(a^{-1}\ovC C) = b+b-\Tr(b) = 0$ and 
	 $A - a^{-1}\ovC \ovB = A - A = 0$. This means that the given $M_x$ acts on the elements
	 of $\J\setminus \J_{17}^{cAB}$ in the following way:
	 \begin{equation*}
	 	\begin{array}{r@{\;\;}c@{\;\;}l}
	 		M_x : (a,b,c\mid A,B,C) & \mapsto & 
	 				(a,0,c \mid 
	 						0, B, 0),
	 	\end{array}
	 \end{equation*}
	 where $a\neq 0$. It is still white, so $B\ovB = ca$.
	 Finally, we act by $L_y''$ with $y = -a^{-1}\ovB$:
	 \begin{equation*}
	 	\begin{array}{r@{\;\;}c@{\;\;}l}
	 		L_y'': (a,0,c\mid 0,B,0) & \mapsto & 
	 			(a,0,0 \mid 0,0,0), 
	 	\end{array}
	 \end{equation*}
	 where $a \neq 0$. In other words, any white point spanned by an element
	 in $\J\setminus\J_{17}^{cAB}$ can be mapped by the action of the stabiliser
	 of $\langle X\rangle$ to the white point spanned by $(1,0,0\mid 0,0,0)$. 
\end{proof}

\begin{lemma}
	\label{lemma:1_white_primitive}
	The action of\/ $\SE_6(F)$ on white points is primitive.
\end{lemma}

\begin{proof}
	From the previous Lemma it follows that if $\OO$ is non-split, then the action of 
	$\SE_6(F)$ on white points is $2$-transitive and hence primitive. It remains to prove
	the statement in the case when $\OO$ is split.

	Suppose $X,Y \in \J$ are white vectors such that $\langle X \rangle \neq
	\langle Y \rangle$. Define $\sim$ to be an $\SE_6(F)$-congruence on white points and
	let $\langle X\rangle \sim \langle Y\rangle$. Our aim is to show that this generates 
	the universal congruence. 
	Since for $\OO$ split the action on the white vectors is transitive, we may assume
	$X = (0,0,1 \mid 0,0,0)$. As mentioned in the beginning of this section, 
	$\langle X\rangle$ 
	determines the $17$-dimensional space $\J_{17}^{cAB}$. We now distinguish
	two cases.
	
	If $Y \in \J_{17}^{cAB}$, then acting by the stabiliser of $\langle X\rangle$ 
	we get $\langle X\rangle \sim \langle \hat{Y}\rangle $ for all white 
	$\hat{Y} \in \J_{17}^{cAB}\setminus \langle X \rangle$. Take
	$\hat{Y} = (0,0,0\mid e_0,0,0) \in \J_{17}^{cAB}$ and 
	\mbox{$\hat{X} = (0,1,0\mid 0,0,0)
	\not\in \J_{17}^{cAB}$}. As we see from the earlier calculations, both
	$X$ and $\hat{X}$ are in the $17$-space determined by $\langle \hat{Y} \rangle$. 
	Acting by the
	stabiliser of $\langle \hat{Y}\rangle$ 
	we map $\langle X\rangle $ to $\langle \hat{X} \rangle $,
	and so ensure $\langle \hat{Y} \rangle \sim \langle \hat{X} \rangle$,
	and so we have the chain $\langle X\rangle \sim \langle\hat{Y}\rangle \sim 
	\langle\hat{X}\rangle$. To get $\langle X\rangle \sim \langle\hat{X}\rangle$ for 
	all white $\hat{X}$ outside $\J_{17}^{cAB}$, we again act by the stabiliser of $\langle
	X\rangle$.
	It follows that $\langle X \rangle$ is congruent to any white point generated by a 
	vector in $\J$, and so we get the universal congruence in this case.
	
	On the other hand, if $Y$ lies outside of $\J_{17}^{cAB}$, then we get $\langle X
	\rangle\sim \langle\hat{Y}\rangle$
	for all white $\hat{Y} \in \J\setminus \J_{17}^{cAB}$ since the stabiliser of 
	$\langle X\rangle$ is transitive on the white points spanned by those. 
	In particular, we may take $\hat{Y} = (1,0,0\mid 0,0,0)$.
	Acting by the stabiliser of $\langle\hat{Y}\rangle$ on both sides in 
	$\langle X\rangle \sim \langle\hat{Y}\rangle$, we map 
	$\langle X\rangle$ to $\langle \hat{X}\rangle$ with $\hat{X} = (0,0,0\mid e_0,0,0)$.
	Note that both $X$ and $\hat{X}$ are not 
	in $\J_{17}^{aBC}$ which is the $17$-space determined by $\hat{Y}$, But 
	$\hat{X} \in \J_{17}^{cAB}$ and by transitivity we get $\langle X\rangle 
	\sim \langle \hat{X} \rangle$. Again,
	we act by the stabiliser of $\langle X\rangle $ to ensure $\langle X\rangle  \sim
	\langle \hat{X} \rangle$ for all white points $\langle \hat{X} \rangle$ spanned by 
	$\hat{X} \in \J_{17}^{cAB}$, i.e. our $\SE_6(F)$-congruence is trivial in this case 
	as well. 
	
\end{proof}

\section{White points in the case of a finite field}

In this section $F$ is a finite field of $q$ elements, that is, $F = \Fq$. 
Our aim is to count the white points in this case. 

\begin{theorem}
	\label{theorem:white_points_count}
	If $F = \Fq$, then there are precisely
	\begin{equation}
		\frac{(q^{12}-1) (q^9-1)}{(q^4-1)} 
	\end{equation}
	white vectors in $\J$. 
\end{theorem}

\begin{proof}
	In the proof we study the series of subspaces
	\begin{equation*}
		0 < \J_{10}^{abC} < \J_{26}^{abABC} < \J.
	\end{equation*}
	
	First, $(a,b,0\mid 0,0,C) \in \J_{10}^{abC}$ is white if and only if 
	$ab - C\ovC = 0$. We notice that $ab - C\ovC$ is a quadratic form of 
	\textit{plus} type defined on $\J_{10}^{abC}$, so there are
	$(q^5-1)(q^4+1)$ white vectors in this subspace. 
	
	Next, suppose $(a,b,c\mid A,B,C) \in \J\setminus \J_{26}^{abABC}$ is white. 
	Then $C = \ovB \ovA c^{-1}$, $b = A\ovA c^{-1}$ and $a = B\ovB c^{-1}$.
	We may choose $A,B,c$ to be arbitrary (with $c \neq 0$), so there are 
	$q^{16}(q-1)$ white vectors in $\J\setminus\J_{26}^{abABC}$.
	
	Finally, we investigate the white vectors $(a,b,0\mid A,B,C) \in 
	 \J_{26}^{abABC}\setminus\J_{10}^{abC}$. The conditions for such a vector 
	 to be white take the following form:
	 \begin{equation}
		\left.\begin{array}{r@{\;}c@{\;}l}
			A \ovA=B\ovB = AB & = & 0, \\
			C \ovC & = & ab, \\
			BC & = & a\ovA, \\
			CA & = & b\ovB.
		\end{array}
		\right\}.
	\end{equation}
	Note that we also require $(A,B) \neq (0,0)$. In case $A = 0$, $B\neq 0$ we apply
	$\delta$ followed by $\tau$ to $(a,b,0\mid A,B,C)$ in order to obtain a 
	vector of the form $(a,b,0\mid A,B,C)$ with $A \neq 0$. Note that the values
	of $a,b,A,B,C$ are not the same as in the initial Albert vector. So, assuming 
	$A \neq 0$, we have $AB = 0$ exactly when $B$ is in a particular $4$-dimensional
	subspace of $\OO$ and any such $B$ satisfies $B\ovB = 0$. For any octonion $x$, the
	action by the element $L_x$ establishes a bijection between the white vectors of
	the form $(*,*,0\mid A,B,*)$ and those of the form \mbox{$(*,*,0\mid A,B+\ovA x,*)$}.
	Left-multiplication by $\ovA$ annihilates a $4$-dimensional subspace of $\OO$
	(see Lemma \ref{lemma:1_octonion_annihilator}), so by
	the rank-nullity theorem we conclude that the image 
	$\mathscr{A} = \{\ \ovA x : x \in \OO\ \}$ is also $4$-dimensional. Note that
	$A (\ovA x) = (A\ovA)x = 0$, for any $x \in \OO$, so $\mathscr{A}$ is the $4$-space
	of all octonions left-annihilated by $A$, and therefore it contains $-B$. 
	Now we pick an octonion $x$ such that $\ovA x = -B$ to obtain the bijection between
	the white vectors of the form $(*,*,0\mid A,B,*)$ with $A\neq 0$ and those of the form 
	$(*,*,0\mid A,0,*)$. An Albert vector $(a,b,0\mid A,0,C)$ is white if and only if
	$A\ovA = C\ovC = CA = 0$ and $a = 0$, with no dependence on $b$. As before, $C$ lies in a
	particular $4$-dimensional subspace of $\OO$, hence $(0,b,0\mid 0,0,C)$ lies in a
	particular $5$-dimensional subspace of $\J_{10}^{abC}$, so for any choice of the 
	pair $(A,B)$ there are $q^5$ white vectors. If $A \neq 0$, then there are
	$(q^4-1)(q^3+1)$ choices for $A$, and for each of these $q^4$ choices for $B$.
	If $A = 0$, we have $(q^4-1)(q^3+1)$ choices for $B$. It follows that in total there
	are
	\begin{equation*}
		q^5 ( q^4(q^4-1)(q^3+1) + (q^4-1)(q^3+1) ) = q^5(q^8-1)(q^3+1)
	\end{equation*}
	white vectors in $\J_{26}^{abABC}\setminus\J_{10}^{abC}$.
	
	The calculations above give the numbers of white vectors in certain subsets of $\J$ as
	shown in the following table.
	
	\begin{table}[h!]
	\begin{tabular}{r|c|c|c}
		subset & $\J_{10}^{abC}$ & $\J_{26}^{abABC}\setminus\J_{10}^{abC}$ & 
										$\J\setminus \J_{26}^{abABC}$ \\ \hline
		\rule{0pt}{10pt}number of white vectors & $(q^5-1)(q^4+1)$ & $q^5(q^8-1)(q^3+1)$ & 
										$q^{16}(q-1)$
	\end{tabular}
	\end{table}
	\noindent Adding these numbers gives the result. 
\end{proof}

\begin{corollary}
	There are precisely
	\begin{equation}
		\frac{(q^{12}-1) (q^9-1)}{(q^4-1)(q-1)} 
	\end{equation}
	white points in the case $F = \Fq$. 
\end{corollary}

\section{The stabiliser of a white point}

\label{section:stabiliser}

In this section we assume that $\OO$ is a split octonion algebra. 
It is our aim now to obtain the stabiliser in $\SE_6(F)$ of a white point. 
In particular, we prove the following result.

\begin{theorem}
    \label{theorem:1_white_stab}
    If $\OO$ is split, then the stabiliser of a white vector in $\SE_6(F)$ is isomorphic to the group
    generated by the actions of the elements 
    $M_x$, $L_x$, $M_x'$ and $L_x''$ on $\J$ 
    as $x$ ranges over $\OO$ and this is a group of shape
    \begin{equation}
        F^{16}\cn\Spin_{10}^+(F).
    \end{equation}
    The stabiliser of a white point is isomorphic to
    
    \begin{equation}
    	F^{16}\cn\Spin_{10}^+(F).F^{\times},
    \end{equation}
    where $F^{\times}$ is the multiplicative group of the field $F$. 
\end{theorem}

This whole section is devoted to proving this result. Some of this proof is in the running text,
and some of it is contained in a series of technical lemmata.  Since it was shown
that the group $\SE_6(F)$ acts transitively on the set of white points, it is sufficient to
study the stabiliser of a specific white vector.  For instance, it is convenient to take
$v = (0,0,1\mid 0,0,0)$. First thing to notice is that $v$ is invariant under the action of the
elements of the form
\begin{equation}
    L_x'' = \begin{bmatrix}
        1&0&x \\
        0&1&0 \\
        0&0&1
    \end{bmatrix},\quad 
    M_y' = \begin{bmatrix}
        1&0&0\\
        0&1&y\\
        0&0&1
    \end{bmatrix},
\end{equation}
where $x,y \in \OO$.

\begin{lemma}$ $
	\label{lemma:1_q16}
	\begin{enumerate}[(a)]
		\item Let $Q$ be any of the $\{L,L',L'',M,M',M''\}$. 
		Then the actions on $\J$ of the 
		elements $Q_x$ where $x$ ranges over $\OO$ generate an elementary abelian group 
		isomorphic to $F^8$. 
		
		\item Let $(R, S)$ be any pair from the set 
		$\{(L,M''), (L',M), (L'',M')\}$ or any of the 
		$\{(L,M'), (L',M''), (L'',M)\}$. Then
		the actions of $R_x$ and $S_x$, as $x$ ranges through $\OO$, 
		generate an elementary abelian 
		group isomorphic to $F^{16}$. 
	\end{enumerate}
\end{lemma}

\begin{proof}
	To show part (a) for the elements $L_x,L_x',L_x''$ it is enough to consider just, say, 
	$L_x''$ as to obtain the result for the rest of them we can apply the action of the 
	triality element $\tau$.  Similarly, out of $M_x,M_x',M_x''$ we only need to consider, for instance, $M_x'$.
	The actions of $L_x''$ and $M_y'$ on $\J$ are given by
    \begin{equation}
    	\label{eq:1_q16_action}
        \begin{array}{r@{\;\;}c@{\;\;}l}
            L_x'': (a,b,c\mid A, B, C) & \mapsto &
            (a,b,c+a x \ovx  + \Tr(Bx) \mid A+\ovC x, B+A \ovx , C), \\
            M_y': (a,b,c\mid A,B,C) & \mapsto &
            (a,b,c+b y \ovy  + \Tr(\ovy A) \mid A + By, B+\ovy \ovC , C).
        \end{array}
    \end{equation}
    We notice that the action is nontrivial whenever $x$ and $y$ are non-zero. 
    The element $M_y'$ sends $(a,b,c+a x \ovx  + \Tr(Bx) \mid A+\ovC x, B+a
    \ovx , C)$ to
    \begin{equation*}
        (a,b,c+ax\ovx +\Tr(Bx)+by\ovy  + \Tr(\ovy A) \mid
            A+\ovC x + by, B+a\ovx  + \ovy \ovC , C ),
    \end{equation*}
    and the element $L_x''$ sends $(a,b,c+b y \ovy  + \Tr(\ovy A) \mid A + By, B+\ovy 
    \ovC , C)$ to
    \begin{equation*}
        (a,b,c+by\ovy  + \Tr(\ovy A) +ax\ovx  +\Tr(Bx) \mid A+by+\ovC x, B+\ovy \ovC 
        +a \ovx , C).
    \end{equation*}
    Hence, the actions of these elements commute. Similarly, it is straightforward to verify that
    the actions of $L_x''$ and $L_y''$ commute as well as the actions of $M_x'$ and $M_y'$.
    Moreover, the element $L_y''$ sends $(a,b,c+a x \ovx  + \Tr(Bx) \mid A+\ovC x, B+a
     \ovx , C)$ to
    \begin{equation*}
        (a,b,c+ax\ovx  + \Tr(Bx) + ay\ovy  + \Tr(By) + a\Tr(\ovx y) \mid
            A+\ovC x + \ovC y, B + a\ovx  + a\ovy , C),
    \end{equation*}
    and $L_{x+y}''$ sends $(a,b,c\mid A,B,C)$ to
    \begin{equation*}
        (a,b,c+ax\ovx  + a\Tr(x\ovy ) + ay\ovy  + \Tr(B(x+y)) \mid
        A + \ovC (x+y), B+a(\ovx  + \ovy ), C),
    \end{equation*}
    so the action of $L_{x+y}''$ is the same as the product of the actions of $L_x''$ and
    $L_y''$. A similar calculation shows that the action of $M_{x+y}'$ is the same
    as the product of the actions of $M_x'$ and $M_y'$. It follows that the action of 
    $L_x''$ on $\J$, $x\in \OO$ generates an abelian group $(F^8,+)$ as well as the action
    of the element $M_y'$, $y \in \OO$. We simply denote the abelian group $(F^n,+)$ as $F^n$ in
    our further discussion. 
    
    To prove part (b) we need to verify that the intersection of the corresponding abelian groups, isomorphic to $F^8$ and generated by the actions of $L_x''$ and $M_x'$ is trivial. Suppose that the
    actions of $L_x''$ and $M_y'$ are equal. Then, according to (\ref{eq:1_q16_action}), 
    in the fourth ``co{\"o}rdinate'' we have
    \begin{equation*}
    	 A+\ovC x = A + B y
    \end{equation*}
    for arbitrary $A,B,C \in \OO$. In other words, we get $\ovC x = B y$ for arbitrary octonions
    $B$ and $C$. In particular, if $B = 1_{\OO}$ and $C = 0$, we get $y = 0$ and if
    $B = 0$ and $C = 1_{\OO}$ we obtain $x = 0$. So, the intersection of two copies of
    $F^8$ consists of the identity element, as needed, and the result follows. 
    Again, to get (b) for the rest of the pairs in the first set we apply the 
    triality element. The calculations for the second set of pairs are of the same nature. 
\end{proof}

The next observation is that our white vector $v$ is also invariant under the action of the
elements
    \begin{equation}
        M_x = \begin{bmatrix}
            1&x&0\\
            0&1&0\\
            0&0&1
        \end{bmatrix},\quad 
        L_y = \begin{bmatrix}
            1&0&0\\
            y&1&0\\
            0&0&1
        \end{bmatrix}.
    \end{equation}
First, we show that the actions of these on $\J_{10}^{abC}$ generate a group 
of type $\Omega_{10}^+(F)$. As we will see further, instead of arbitrary octonions
it is enough for $x$ to range through the scalar multiples of the basis elements $e_i$.
Define the quadratic form $\QQ_{10}$ on $\J$ via
\begin{equation}
	\QQ_{10}((a,b,c\mid A,B,C)) = ab - C\ovC. 
\end{equation}
We notice that $\QQ_{10}$ is of \textit{plus} type, so for convenience we
denote $\Omega_{10}(F,\QQ_{10})$ as $\Omega_{10}^+(F)$. 

To construct $\Omega_{10}^+(F)$ we follow the series of steps. First, we consider the
$4$-space $V_4$ spanned by the Albert vectors of the form 
\mbox{$(a,b,0\mid 0,0,C_{-1}e_{-1} + C_1 e_1)$}. 

\begin{lemma}
	The actions of the elements $M_{\lambda e_{\pm 1}}$ and $L_{\lambda e_{\pm 1}}$
	on $V_4$, where $\lambda \in F$, generate a group of type $\Omega_4^+(F)$.
\end{lemma}

\begin{proof}
	Consider $\mathcal{B} = \{v_1, v_2, v_3, v_4\}$ to be the basis of $V_4$ with 
	\begin{equation*}
		\begin{array}{r@{\;}c@{\;}l}
			v_1 & = & (1,0,0 \mid 0,0,0), \\
			v_2 & = & (0,0,0 \mid 0,0,e_1), \\
			v_3 & = & (0,0,0 \mid 0,0,e_{-1}), \\
			v_4 & = & (0,1,0 \mid 0,0,0).
		\end{array}
	\end{equation*}
	With respect to $\mathcal{B}$,  the actions on $V_4$ of the elements $M_{\lambda e_{-1}}$, $M_{\lambda e_1}$, $L_{- \lambda e_{-1}}$, and $L_{- \lambda e_1}$ 
	can be written as $4\times 4$ matrices
	\begin{equation*}
		\begin{bmatrix}
			1 & 0 & \lambda & 0 \\
			0 & 1 & 0 & \lambda \\
			0 & 0 & 1 & 0 \\
			0 & 0 & 0 & 1
		\end{bmatrix} = 
		\begin{bmatrix}
			1 & \lambda \\
			0 & 1
		\end{bmatrix} \otimes
		\begin{bmatrix}
			1 & 0 \\
			0 & 1
		\end{bmatrix},\ \ 
		\begin{bmatrix}
			1 & \lambda & 0 & 0 \\
			0 & 1 & 0 & 0 \\
			0 & 0 & 1 & \lambda \\
			0 & 0 & 0 & 1
		\end{bmatrix} = 
		\begin{bmatrix}
			1 & 0 \\
			0 & 1
		\end{bmatrix} \otimes
		\begin{bmatrix}
			1 & \lambda \\
			0 & 1 
		\end{bmatrix},
	\end{equation*}
	\begin{equation*}
		\begin{bmatrix}
			1 & 0 & 0 & 0 \\
			\lambda & 1 & 0 & 0 \\
			0 & 0 & 1 & 0 \\
			0 & 0 & \lambda & 1
		\end{bmatrix} = 
		\begin{bmatrix}
			1 & 0 \\
			0 & 1
		\end{bmatrix} \otimes
		\begin{bmatrix}
			1 & 0 \\
			\lambda & 1
		\end{bmatrix},\ \ 
		\begin{bmatrix}
			1 & 0 & 0 & 0 \\
			0 & 1 & 0 & 0 \\
			\lambda & 0 & 1 & 0 \\
			0 & \lambda & 0 & 1
		\end{bmatrix} = 
		\begin{bmatrix}
			1 & 0 \\
			\lambda & 1
		\end{bmatrix} \otimes
		\begin{bmatrix}
			1 & 0 \\
			0 & 1
		\end{bmatrix},
	\end{equation*}
	respectively. Here $\otimes$ denotes the Kronecker product.  As we know, 
	\begin{equation*}
		\left\langle
			\begin{bmatrix}
				1 & \lambda \\
				0 & 1
			\end{bmatrix},\ 
			\begin{bmatrix}
				1 & 0 \\
				\lambda & 1
			\end{bmatrix}\ 
			\bigg|
			\ 
			\lambda \in F
		\right\rangle \cong \SL_2(F).
	\end{equation*}
	It follows that
	\begin{multline*}
		\left\langle
			\begin{bmatrix}
				1 & \lambda \\
				0 & 1
			\end{bmatrix} \otimes
			\begin{bmatrix}
				1 & 0 \\
				0 & 1
			\end{bmatrix},\ 
			\begin{bmatrix}
				1 & 0 \\
				\lambda & 1
			\end{bmatrix}\otimes
			\begin{bmatrix}
				1 & 0 \\
				0 & 1
			\end{bmatrix}\ 
			\bigg|
			\ 
			\lambda \in F
		\right\rangle \cong \\
		\cong 
		\left\langle
			\begin{bmatrix}
				1 & 0 \\
				0 & 1
			\end{bmatrix} \otimes
			\begin{bmatrix}
				1 & \lambda \\
				0 & 1
			\end{bmatrix},\ 
			\begin{bmatrix}
				1 & 0 \\
				0 & 1
			\end{bmatrix} \otimes
			\begin{bmatrix}
				1 & 0 \\
				\lambda & 1
			\end{bmatrix}\ 
			\bigg|
			\ 
			\lambda \in F
		\right\rangle \cong \SL_2(F),
	\end{multline*}
	and since 
	\begin{equation*}
		\begin{bmatrix}
				-1 & 0 \\
				0 & -1
		\end{bmatrix} \otimes
		\begin{bmatrix}
				-1 & 0 \\
				0 & -1
		\end{bmatrix} = 
		\begin{bmatrix}
			1 & 0 & 0 & 0 \\
			0 & 1 & 0 & 0 \\
			0 & 0 & 1 & 0 \\
			0 & 0 & 0 & 1
		\end{bmatrix},
	\end{equation*}
	we finally get
	\begin{multline*}
		\left\langle
			\begin{bmatrix}
				1 & 0 \\
				0 & 1
			\end{bmatrix} \otimes 
			\begin{bmatrix}
				1 & \lambda \\
				0 & 1
			\end{bmatrix},\ 
			\begin{bmatrix}
				1 & 0 \\
				0 & 1
			\end{bmatrix} \otimes
			\begin{bmatrix}
				1 & 0 \\
				\lambda & 1
			\end{bmatrix}, \right. \\
			\left. \begin{bmatrix}
				1 & \lambda \\
				0 & 1
			\end{bmatrix} \otimes 
			\begin{bmatrix}
				1 & 0 \\
				0 & 1
			\end{bmatrix},\ 
			\begin{bmatrix}
				1 & 0 \\
				\lambda & 1
			\end{bmatrix} \otimes
			\begin{bmatrix}
				1 & 0 \\
				0 & 1
			\end{bmatrix}\ 
			\bigg|
			\ 
			\lambda \in F
		\right\rangle \cong \SL_2(F) \circ \SL_2(F).
	\end{multline*}
	Now, $\SL_2(F) \circ \SL_2(F) \cong \Omega_4^+(F)$ (see, for example
    \cite{Taylor}), and this finishes the proof. 
\end{proof}

In our construction we use the results of section \ref{section:orthogonal}. Consider the \mbox{$6$-space} $V_6$ spanned by the Albert vectors
$(a,b,0\mid 0,0,C)$, where $C \in \langle e_{-1}, e_{\ombb}, e_{-\ombb}, 
e_1 \rangle$. Our copy of $\Omega_4^+(F)$ preserves two isotropic Albert vectors in
 $V_6$:
\begin{equation}
	\begin{array}{r@{\;}c@{\;}l}
		u_{\ombb} & = & (0,0,0\mid 0,0,e_{\ombb}), \\
		u_{-\ombb} & = & (0,0,0\mid 0,0,e_{-\ombb}). \\
	\end{array} 
\end{equation}
The element $M_{e_{\ombb}}$ preserves $u_{\ombb}$ but not $u_{-\ombb}$.
Therefore, adjoining $M_{e_{\ombb}}$ to $\Omega_4^+(F)$, 
we obtain a subgroup
of $V_4\cn\Omega_4^+(F)$ (Lemma \ref{lemma:1_stabiliser_omega}), and since 
$\Omega_4^+(F)$
is maximal in the latter (Theorem \ref{theorem:1_omega_maximal}), we conclude that the action of $M_{\lambda e_{\pm 1}}$, 
$L_{\lambda e_{\pm 1}}$ and $M_{e_{\ombb}}$ on $V_6$ is that of 
$V_4\cn\Omega_4^+(F)$. That is, we have constructed the group 
$V_4\cn\Omega_4^+(F)$
as the stabiliser of $u_{\ombb}$ in $\Omega_6^+(F)$. 
Now we use the result of Theorem \ref{theorem:1_space_stab}. The element
 $M_{e_{-\ombb}}$ preserves $V_6$ but it does not preserve
$u_{\ombb}$, and as a consequence it does not preserve the $1$-space $\langle
u_{\ombb} \rangle$. Therefore, if we adjoin $M_{e_{-\ombb}}$ to our copy of 
$V_4\cn\Omega_4^+(F)$, we get the action of the group $\Omega_6^+(F)$ on $V_6$. 

Similarly, we consider the $8$-space $V_8$ spanned by the vectors 
$(a,b,0\mid 0,0,C)$ with $C \in \langle e_{-1}, e_{\ombb}, e_{\omega}, 
e_{- \omega}, e_{-\ombb}, e_1 \rangle$. Consider two isotropic Albert vectors
\begin{equation}
	\begin{array}{r@{\;}c@{\;}l}
		u_{\omega} & = & (0,0,0 \mid 0,0,e_{\omega}), \\
		u_{-\omega} & = & (0,0,0 \mid 0,0,e_{-\omega}),
	\end{array}
\end{equation}
which are fixed by our copy of $\Omega_6^+(F)$. The action of the element
 $M_{e_{\omega}}$
on $V_8$ preserves $u_{\omega}$ but not $u_{-\omega}$ and therefore adjoining this element
to $\Omega_6^+(F)$ we get the action of the group $V_6\cn\Omega_6^+(F)$. Next, 
the element $M_{e_{-\omega}}$ does not preserve the $1$-space $\langle u_{\omega} \rangle$,
so appending it to $V_6\cn\Omega_6^+(F)$ we get the action of the group
$\Omega_8^+(F)$ on $V_8$. 

Finally, we consider the $10$-space $\J_{10}^{abC}$ with two isotropic Albert vectors
\begin{equation}
	\begin{array}{r@{\;}c@{\;}l}
		u_0 & = & (0,0,0\mid 0,0,e_0), \\
		u_{-0} & = & (0,0,0\mid 0,0,e_{-0}).
	\end{array}
\end{equation}
Following the same procedure, we adjoin the element $M_{e_0}$ which fixes $u_0$ but not 
$u_{-0}$ to get the action of the group of shape $V_8\cn\Omega_8^+(F)$. 
Appending the action of $M_{e_{-0}}$, which does not preserve $\langle u_0 \rangle$, to
this yields the action of $\Omega_{10}^+(F)$ on $\J_{10}^{abC}$. 
Lemma \ref{lemma:1_q16} allows us to conclude that we have shown the following result.

\begin{lemma}
	\label{lemma:1_omega10plus}
	The actions of $M_x$ and $L_x$ on $\J_{10}^{abC}$ generate the 
	group $\Omega_{10}^+(F)$ as $x$ ranges through $\OO$. 
\end{lemma}

Now we need to understand the action of the elements $M_x$ and $L_x$ on the whole 
$27$-space $\J$.  

\begin{lemma}
	\label{lemma:1_2_actions}
	Suppose an element $g$ of the stabiliser in $\SE_6(F)$ of $v$ preserves 
	the decomposition of the Albert space into
	the direct sum of the form $\J = \J_{1}^c \oplus \J_{16}^{AB} \oplus \J_{10}^{abC}$.
	
	\begin{enumerate}[(a)]
	\item If the action of $g$ on the $10$-space $\J_{10}^{abC}$ 
	is given by
	\begin{equation*}
		\begin{array}{r@{\;\;}c@{\;\;}l}
			(1,0,0\mid 0,0,0) & \mapsto & (\lambda,0,0\mid 0,0,0), \\
			(0,1,0\mid 0,0,0) & \mapsto & (0,\lambda^{-1},0\mid 0,0,0), \\
			(0,0,0\mid 0,0,C) & \mapsto & (0,0,0 \mid 0,0,C),
		\end{array}
	\end{equation*}
	then $\lambda$ is a square in $F$. 
	
	\item On the other hand, an action of the type 
	\begin{equation*}
		\begin{array}{r@{\;\;}c@{\;\;}l}
			(1,0,0\mid 0,0,0) & \mapsto & (0,\lambda,0\mid 0,0,0), \\
			(0,1,0\mid 0,0,0) & \mapsto & (\lambda^{-1},0,0\mid 0,0,0), \\
			(0,0,0\mid 0,0,C) & \mapsto & (0,0,0 \mid 0,0,C)
		\end{array}
	\end{equation*}
	is impossible. 
	
	\item Finally, if the action on the $10$-space is 
	trivial, then the action on the corresponding $16$-space is that of $\pm \mathrm{I}_{16}$ 
	(hence, the action on $\J$ is that of $P_{\pm 1}$). 
	
	\end{enumerate}
\end{lemma}

\begin{proof}
	We are considering the elements that fix $\J_{8}^C$ pointwise
    and either fix or swap the $1$-dimensional spaces $\J_1^a$ and $\J_1^b$. 
    So we may assume that these elements respectively fix or swap the corresponding
    $17$-spaces $\J_{17}^{aBC}$ and $\J_{17}^{bAC}$.
    In particular, their intersection, i.e. the space $\J_{8}^C$ is fixed. If the action of the
     stabiliser
    swaps $\J_1^a$ and $\J_1^b$ while leaving the $1$-space $\J_1^c$ in its
    place, then it also swaps the $8$-spaces $\J_8^A$ and
    $\J_8^B$ as these subspaces are the intersections of the $17$-space
    $\J_{17}^{cAB}$ with $\J_{17}^{bAC}$ and $\J_{17}^{aBC}$ respectively.
    
   	Suppose now that an element $g$ in the stabiliser acts in the following manner:
    \begin{equation*}
            g:(a,b,c\mid A,B,C) \mapsto (\lambda a, \lambda^{-1} b, c \mid
                A \phi, B\psi, C),
    \end{equation*}
    where $\phi, \psi : \OO \rightarrow \OO$ are invertible $F$-linear maps.
    As this action is supposed to  preserve the determinant, it has to preserve the cubic term
    $\Tr(ABC)$ in particular, i.e. we must have $\Tr(ABC) = \Tr((A\phi) (B\psi) C)$ for all 
    $A,B,C \in \OO$. This
    is equivalent to the condition $AB = (A\phi) (B\psi)$ for all $A,B \in \OO$, since the 
    original identity is equivalent to
    $\langle AB, \ovC  \rangle = \langle (A\phi)(B\psi) , \ovC  \rangle$.
    By Lemma \ref{lemma:1_white_phipsi2} we find that $A \phi = \mu^{-1} A$ and
    $B \psi = \mu B$ for all $A,B \in \OO$ and some non-zero $\mu \in F$. The individual terms in
     the determinant are being
    changed in the following way:
    \begin{equation*}
        \begin{array}{r@{\;\;}c@{\;\;}l}
        abc & \mapsto & abc, \\
        a A \ovA  & \mapsto & \lambda \mu^{-2} a A \ovA , \\
        b B \ovB  & \mapsto & \lambda^{-1} \mu^2 b B \ovB , \\
        c C \ovC  & \mapsto & c C \ovC ,\\
        \Tr(ABC) & \mapsto & \Tr(ABC).
    \end{array}
    \end{equation*}
    It follows that in order to preserve the determinant we must have $\lambda^{-1} \mu^2 = 1$,
    i.e. $\lambda = \mu^2$.
    
    In case when $g$ acts as
    \begin{equation*}
   		g: (a,b,c\mid A,B,C) \mapsto (\lambda^{-1} b, \lambda a, c \mid B\psi, A\phi, C),
   	\end{equation*}
    we get $\Tr(ABC) = \Tr((B\psi) (A\phi) C)$
    for all $A,B,C \in \OO$. This holds if and only if $AB = (B\psi) (A\phi)$ for
    all $A,B \in \OO$. Lemma
    \ref{lemma:1_white_phipsi1} asserts that there are no such maps $\phi$ and $\psi$,
    and so this rules out the latter case.
    
    Finally, if we assume the trivial action on $\J_{10}^{abC}$, then we get $\lambda = 1$,
     i.e. 
    $\mu^2 = 1$, so the action on $\J$ is indeed that of $P_{\pm 1}$. 
\end{proof}

Now let $X = (a,b,c\mid A,B,C)$
and let $Y = (a',b',c' \mid A', B', C')$. An isometry which maps $X$ to
    $Y$ and $v$ to $\lambda v$ must send $\fdet(X+v)-\fdet(X) = ab-C\ovC $ to
    $\fdet(Y+\lambda v) - \fdet(Y) = \lambda (a'b' - C'\ovC ')$. The \mbox{$17$-dimensional}
     radical
    of both of these forms is fixed, and the quadratic form $ab-C\ovC$ is being scaled by
    a factor of $\lambda$. In particular, when $\lambda = 1$, the quadratic
    form is being preserved. So, the action of the vector stabiliser on the $10$-dimensional
    quotient is that of a subgroup of $\GO_{10}^+(F)$. 

    Consider the white vectors of the form $(a,0,c\mid A,B,0)$ and $(0,b,c\mid A,B,0)$ with
    $a,b \neq 0$. In the first case the conditions for being white are
    \begin{equation*}
        \left.
            \begin{array}{r@{\;}c@{\;}l}
                A\ovA  & = & 0,\\
                B\ovB  & = & ac, \\
                a\ovA  & = & 0, \\
                AB & = & 0.
            \end{array}
        \right\}
    \end{equation*}
    In other words, we have a white vector of the form
    $(a,0,B\ovB /a\mid 0,B,0)$. For the second vector we get
    \begin{equation*}
        \left.
            \begin{array}{r@{\;}c@{\;}l}
                bc & = & A\ovA , \\
                B\ovB  & = & 0, \\
                b\ovB  & = & 0,
            \end{array}
        \right\}
    \end{equation*}
    so the vector has the form $(0,b,A\ovA /b\mid A,0,0)$. The elements $M_x'$ and $L_x''$
    transform these in the following way:
    \begin{equation*}
        \begin{array}{r@{\;\;}c@{\;\;}l}
            M_x': (a,0,B\ovB /a\mid 0,B,0) & \mapsto &
            (a,0,B\ovB /a\mid 0,B,0), \\
            M_x': (0,b,A\ovA /b\mid A,0,0) & \mapsto &
            (0,b,A\ovA /b+bx\ovx  + \Tr(\ovx A) \mid A+bx,0,0), \\
            L_x'': (a,0,B\ovB /a\mid 0,B,0) & \mapsto &
            (a,0,B\ovB /a+ax\ovx +\Tr(Bx)\mid 0,B+a\ovx , 0), \\
            L_x'': (0,b,A\ovA /b\mid A,0,0) & \mapsto &
            (0,b,A\ovA /b\mid A,0,0).
        \end{array}
    \end{equation*}
    Note that we already have an elementary abelian group $F^{16}$ acting
    on the $17$-space \mbox{$\J_{17}^{cAB}$}. We can now 
    invoke Lemma \ref{lemma:1_2_actions} to conclude that the action of the stabiliser
    on the remaining $10$-space $\J_{10}^{abC}$ is that of $\Omega_{10}^+(F)$ and the 
    kernel of the action on $\J$ has order no more than two. 

\begin{theorem}
	\label{theorem:1_splin10plus}
	The actions of the elements $M_x$ and $L_x$ on $\J$ where $x$ ranges through a 
	split octonion algebra $\OO$
	generate a group of type $\Spin_{10}^+(F)$ understood as $\Omega_{10}^+(F)$
	in case of characteristic $2$. 
\end{theorem}

With the result of Lemma \ref{lemma:1_q16} we conclude that the stabiliser of a white
vector is indeed a group of shape
$F^{16}\cn\Spin_{10}^+(F)$ as usual understood as $F^{16}\cn\Omega_{10}^+(F)$ in case of
characteristic $2$. 

Now we have enough ingredients to produce the vector stabiliser. As before, we consider the stabiliser of the white
    vector $v = (0,0,1\mid 0,0,0)$.
As we know from Theorem \ref{theorem:1_splin10plus} and Lemma \ref{lemma:1_2_actions}, 
the actions of the elements $M_x$ and $L_x$ on $\J$ generate a group of type 
$\Spin_{10}^+(F)$. It is easy to check
 that this copy of $\Spin_{10}^+(F)$ normalises the elementary abelian
group $F^{16}$ from Lemma \ref{lemma:1_q16}. A straighforward computation
illustrates the following result:

\begin{equation}
	\begin{array}{r@{\;}c@{\;}l}
		(M_x')^{L_y} & \text{ acts as } & M_x', \\
		(M_x')^{M_y} & \text{ acts as } &  L_{-yx}'' \cdot M_x', \\
		(L_x'')^{L_y} & \text{ acts as } & M_{-yx}' \cdot L_x'', \\
		(L_x'')^{M_y} & \text{ acts as } & L_x'', \\
	\end{array}
\end{equation}
where the products in the right-hand side are understood as the products of the actions rather
than as the matrix products. Furthermore, the intersection
of the groups $\Spin_{10}^+(F)$ and $F^{16}$ is trivial: the action of $\Spin_{10}^+(F)$
preserves the decomposition $\J = \J_{1}^c \oplus \J_{16}^{AB} \oplus \J_{10}^{abC}$,
while any non-trivial action of the elementary abelian group $F^{16}$ fails to do so.
Indeed, a general element in $F^{16}$ has the form $M_x' \cdot L_y''$ for some $x,y \in 
\OO$ and it sends an Albert vector $(a,b,c\mid A,B,C)$ to 
\begin{equation*}
	(a,b,c+a\NN(y)+b\NN(x) + \Tr(By) + \Tr(\ovx A) + \Tr(\ovx \ovC y) \mid
			A+\ovC y + bx, B+a\ovy + \ovx \ovC, C ).
\end{equation*}
So, we have shown that the actions of the elements $M_x', L_x'', M_x, L_x$ on $\J$ generate
a group of shape $F^{16}\cn\Spin_{10}^+(F)$, as $x$ ranges through a split algebra $\OO$.

Next, we consider the white point $\langle v \rangle$ spanned by our white vector. The stabiliser
in $\SE_6(F)$ of $\langle v \rangle$, where $v = (0,0,1\mid 0,0,0)$, maps 
$v$ to $\lambda v$ for some non-zero 
$\lambda \in F$. For instance, this can be achieved by the elements 
\begin{equation}
	\label{eq:puprime}
	P_{u^{-1}}' = \diag(1_{\OO},u^{-1},u) = \begin{bmatrix}
		1 & 0 & 0 \\
		0 & u^{-1} & 0 \\
		0 & 0 & u
	\end{bmatrix}	 
\end{equation}
with 
$u$ being an invertible octonion of arbitrary norm. 
Indeed, any such element $P_{u^{-1}}'$ sends
$(0,0,1\mid 0,0,0)$ to $(0,0,\NN(u)\mid 0,0,0)$ and since $\NN(u)$ can be any non-zero 
field element, we get an abelian group $F^{\times}$ on top of the vector stabiliser. This finishes
the proof of the main theorem in this section. 

Now, since the vector stabiliser is generated by the actions of $M_x$, $L_x$, $M_x'$,
$L_x''$ on $\J$, and the subgroup of $\SE_6(F)$ generated by 
$M_x,M_x',M_x'',L_x,L_x',L_x''$ acts transitively on the white points, 
we make the following conclusion. 

\begin{theorem}
The group $\SE_6(F)$ is generated by the actions
of $M_x,M_x',M_x''$ and $L_x,L_x',L_x''$ on $\J$ as $x$ ranges through $\OO$. 
\end{theorem}

If $F = \Fq$, the stabiliser of a white vector is a group of shape $q^{16}\cn\Spin_{10}^+(q)$, while the stabiliser
of a white point is a group of shape $q^{16}\cn\Spin_{10}^+(q).\CC_{q-1}$. As a
 consequence, we now have:
\begin{equation}
	|\SE_6(q)| = q^{36}(q^{12}-1)(q^9-1)(q^8-1)(q^6-1)(q^5-1)(q^2-1).
\end{equation}

We obtain $\EE_6(q)$ as the quotient of $\SE_6(q)$ by any scalars it contains.
Note that $\SE_6(q)$ contains non-trivial scalars if and only if $q \equiv 1 \Mod{3}$, so
\begin{equation}
	|\EE_6(q)| = \frac{1}{\gcd(3,q-1)} 
				q^{36}(q^{12}-1)(q^9-1)(q^8-1)(q^6-1)(q^5-1)(q^2-1).
\end{equation}

\section{Simplicity of $\EE_6(F)$} 

The classical way of showing the simplicity of certain groups is the following lemma.

\begin{lemma}[Iwasawa]
    If $G$ is a perfect group acting faithfully and primitively 
    on a set $\Omega$, and the point stabilizer $H$ has a normal
    abelian subgroup $A$ whose conjugates generate $G$, then
    $G$ is simple.
\end{lemma}

First, we show that the subgroup of $\SE_6(F)$ stabilising 
all the white points simultaneously acts on $\J$ by scalar multiplications, and hence
the action of $\E_6(F)$ on the set of white points is faithful.

\begin{lemma}
	The subgroup in $\SE_6(F)$ stabilising simultaneously all white points
	is the group of scalars.  
\end{lemma}

\begin{proof}
 Consider the action of
this stabiliser on $\J_{10}^{abC}$ and pick the basis
\begin{equation}
	\begin{array}{r@{\;}c@{\;}l}
		v_1 & = & (1,0,0\mid 0,0,0), \\
		v_2 & = & (0,1,0\mid 0,0,0), \\
		v_{i+2} & = & (a_i, b_i,0\mid 0,0,C_i),
	\end{array}
\end{equation} 
where $1 \leqslant i \leqslant 8$ and $C_i \ovC_i = a_i b_i$. Since in particular we stabilise
$\langle v_1 \rangle,\ldots,\langle v_{10}\rangle$, the action on $\J_{10}^{cAB}$ is that of a 
$10 \times 10$ diagonal matrix $\diag(\lambda_1, \ldots, \lambda_{10})$ 
with respect to the basis $\{v_1, \ldots, v_{10}\}$. Consider 
an Albert vector $v = (a,b,0\mid 0,0,C)$, where $C = C_1 + \cdots + C_8$ and 
$a,b$ are such that $v$ is white, i.e. $C\ovC = ab$. Now, if $F\neq \mathbb{F}_2$, we can choose
$a,b \in F$ in such a way that $v$ can be written as a linear combination 
$v = \alpha v_1 + \beta v_2 + v_3 + \cdots + v_{10}$ with $\alpha \neq 0$. The stabiliser of 
all white points maps $v$ to $\lambda v$ for some non-zero $\lambda \in F$, so this ensures
that $\lambda = \lambda_1 = \lambda_3 = \cdots = \lambda_{10}$. We now adjust the chosen values
of $a$ and $b$ to obtain a linear combination with $\beta \neq 0$, and so $\lambda = 
\lambda_2 = \lambda_3 = \cdots \lambda_{10}$. It follows that the action on $\J_{10}^{abC}$ 
is just the multiplication by $\lambda$. 

When $F = \mathbb{F}_2$, we take $\OO$ to be the split octonion algebra with 
our favourite basis $\{ e_i\ \mid\ i \in \pm \{0,1,\omega,\ombb\}\}$. For the 
$10$-space $\J_{10}^{abC}$ we choose the basis 
\begin{equation}
	\begin{array}{r@{\;}c@{\;}l}
		v_1 & = & (1,0,0\mid 0,0,0), \\
		v_2 & = & (0,1,0\mid 0,0,0), \\
		v_{i+2} & = & (0,0,0\mid 0,0,e_i),
	\end{array}
\end{equation} 
and then proceed in the same manner. The vector $v = v_1 + \cdots + v_{10}$ is white and since 
there is a single choice for a non-zero scalar in $\mathbb{F}_2$, it is being fixed and 
the action on the whole $10$-space in this case is that of $\diag(1,\ldots, 1)$. 

Now, by using the triality element, we map $\J_{10}^{abC}$ to $\J_{10}^{bcA}$ and further
to $\J_{10}^{caB}$ and so we obtain that the stabiliser of all white points
acts on $\J$ by scalar multiplications. That is, the stabiliser of all the white points
is trivial in $\EE_6(F)$. 

\end{proof}

From Lemma \ref{lemma:1_white_primitive} we know that the action of $\EE_6(F)$ 
on the white points is primitive. We need to show that the group is perfect.

\begin{lemma}
	The group $\SE_6(F)$ is perfect.
\end{lemma}

\begin{proof}
	This does not present great difficulties. A very straightforward computation 
	shows that 
	\begin{equation*}
	\begin{array}{r@{\;}c@{\;}l}
		(L_{-1}'')^{-1} \cdot L_x' \cdot L_{-1}'' \cdot (L_x')^{-1} & 
								\text{ acts as } & M_x, \\
		(L_{-1})^{-1} \cdot L_x'' \cdot L_{-1} \cdot (L_x'')^{-1} & 
								\text{ acts as } & M_x', \\
		(L_{-1}')^{-1} \cdot L_x \cdot L_{-1}' \cdot (L_x)^{-1} & 
								\text{ acts as } & M_x'', \\
								
		(M_{-1}')^{-1} \cdot M_x'' \cdot M_{-1}' \cdot (M_x'')^{-1} & 
								\text{ acts as } & L_x, \\
		(M_{-1}'')^{-1} \cdot M_x \cdot M_{-1}'' \cdot (M_x)^{-1} & 
								\text{ acts as } & L_x', \\
		(M_{-1})^{-1} \cdot M_x' \cdot M_{-1} \cdot (M_x')^{-1} & 
								\text{ acts as } & L_x'', \\
	\end{array}
	\end{equation*}
	where as before $A \cdot B$ is understood as the product of the actions 
	by the matrices $A$ and $B$. Hence, every generator is in fact a commutator.
\end{proof}
Finally, using the Iwasawa's Lemma we obtain the following theorem.

\begin{theorem}
	The group $\EE_6(F)$ is simple. 
\end{theorem}

\section{Some related geometry}

\label{section:1_geometry}

In this section we discuss some of the underlying geometry related to white 
points. Consider first a $10$-dimensional
space $\J_{10}^{abC}$ and note that it 
contains only white and grey vectors. In this section we are interested in finding the 
stabiliser of $\J_{10}^{abC}$, discovering some of its properties, and also finding
the joint stabiliser of such a $10$-space and a white point. 
Note that throughout the whole section $\OO$ is a split 
octonion algebra. 

First, we shall be interested in the stabiliser in $\SE_6(F)$ of $\J_{10}^{abC}$. 
The following lemma helps to get an idea what the stabiliser
we are looking for can be. 

\begin{lemma}
	\label{lemma:1_white_space_stab}
	The stabiliser in $\SE_6(F)$ of $\J_{10}^{abC}$ contains
	a subgroup of shape 
	\begin{equation}
		F^{16}\cn\Spin_{10}^+(F).F^{\times}.
	\end{equation}
\end{lemma}

\begin{proof}
	We take an arbitrary vector $(a,b,0 \mid 0,0,C)$ in 
	$\J_{10}^{abC}$ and look how the elements $M_x$, $M_x'$,
	$M_x''$, $L_x$, $L_x'$, and $L_x''$ act on it:
	\begin{equation*}
		\begin{array}{r@{\;}c@{\;}l}
			M_x: (a,b,0\mid 0,0,C) & \mapsto & 
				(a, b + a\NN(x) + \Tr(\ovx C), 0 \mid
					0,0, C + a x), \\
			
			M_x': (a,b,0\mid 0,0,C) & \mapsto & 
				(a, b, a \NN(x) \mid b x, \ovx \ovC, C), \\
				
			M_x'': (a,b,0 \mid 0,0,C) & \mapsto & 
				(a, b, 0 \mid 0, 0, C), \\
				
			L_x: (a,b,0 \mid 0,0,C) & \mapsto & 
				(a + b\NN(x) + \Tr(Cx), b, 0 \mid
					0, 0, C + b\ovx), \\
					
			L_x': (a,b,0 \mid 0,0,C) & \mapsto & 
				(a, b, 0 \mid 0,0,C), \\
				
			L_x'': (a,b,0 \mid 0,0,C) & \mapsto & 
				(a, b, a\NN(x) + \Tr(B x) \mid
					\ovC x, a\ovx, C).
		\end{array}
	\end{equation*}
	It is visibly clear now that the elements $M_x$, $L_x$,
	$M_x''$, $L_x'$ preserve $\J_{10}^{abC}$. We have been in 
	a similar situation before 
	(Theorem \ref{theorem:1_white_stab}), so we just note here
	that the abelian group $F^{16}$ generated by the actions of
	$M_x''$ and $L_x'$ is different from the one generated by
	$M_x'$ and $L_x''$ in the theorem we refer to. Thus, 
	the stabiliser of $\J_{10}^{abC}$ contains a group of
	shape $F^{16}\cn\Spin_{10}^+(F)$. 
	
	Finally, we notice that each of the elements $P_u$, $P_u'$,
	$P_u''$ preserve $\J_{10}^{abC}$. It is straightforward to see
	that for any invertible octonion $u$, $P_u \cdot P_u' \cdot
	P_u''$ acts on $\J$ as the identity matrix. Therefore we only
	consider the action on $\J_{10}^{abC}$ of two of them, say
	$P_u$ and $P_u'$, and since $P_u$ acts on $\J$ as
	$M_{u-1} \cdot L_1 \cdot M_{u^{-1}-1} \cdot L_{-u}$, 
	we conclude that they represent elements of 
	$\Spin_{10}^+(F)$, which we already have as a part of the 
	stabiliser, so it is enough to consider the elements $P_u'$.
	
	Note that the elements $P_u'$, where $u$ is an arbitrary
	invertible octonion, preserve $\J_{10}^{abC}$:
	\begin{equation*}
		\begin{array}{r@{\;}c@{\;}l}
			P_u': (a,b,0\mid 0,0,C) & \mapsto &
		(a,b \NN(u), 0 \mid 0,
			0, C u), 
		\end{array}
	\end{equation*}
	so we get $F^{\times}$ on top of the stabiliser, 
	and the result follows.
\end{proof}

The following theorem strengthens this result:
we prove that the stabiliser of $\J_{10}^{abC}$ is 
precisely a group of shape $F^{16}\cn\Spin_{10}^+(F).F^{\times}$.

\begin{theorem}
	\label{theorem:1_white_space_stab}
	The stabiliser in $\SE_6(F)$ of 
	$\J_{10}^{abC}$ is a subgroup of
	shape
	\begin{equation}
		F^{16}\cn\Spin_{10}^+(F).F^{\times},
	\end{equation}
	generated by the actions on $\J$ of the elements 
	$M_x$, $M_x''$, $L_x$, $L_x'$ as $x$ ranges through $\OO$,
	and $P_u'$ as $u$ ranges through invertible octonions in 
	$\OO$.
\end{theorem}

\begin{proof}
	Consider the white point $W$ spanned by $(0,0,1 \mid 0,0,0)$. We are interested in the joint stabiliser of 
	$\J_{10}^{abC}$ and $W$. Theorem \ref{theorem:1_white_stab} tells us that the stabiliser in $\SE_6(F)$ of 
	$W$ has the shape $G_W = F^{16}\cn\Spin_{10}^+(F).F^{\times}$, and 
	$H \cong  \Spin_{10}^+(F).F^{\times}$, generated by the actions of $M_x$, $L_x$, $P_u'$, stabilises
	$\J_{10}^{abC}$. Note that the stabiliser of $\J_{10}^{abC}$ acts as  
	$\Omega_{10}^+(F).F^{\times}$ on $\J_{10}^{abC}$ since the normal subgroup $F^{16}$ and the 
	central involution (if any) of the standard complement $\Spin_{10}^+(F).F^{\times}$, generated
	by $L_x$, $M_x$, and $P_u'$, acts trivially thereon. 
	
	The normal subgroup $T_1 \cong F^{16}$ of $G_W$ is a left (or right) transversal 
	of $H$ in $G_W$. It is easy to see that no non-trivial element of $T_1$ stabilises $\J_{10}^{abC}$. Indeed, a general
	element in $T_1$ has the form $M_x' \cdot L_y''$ for some $x,y \in \OO$ and it sends \mbox{$(a,b,0 \mid 0,0,C) \in \J_{10}^{abC}$}
	to
	\begin{equation*}
		(a,b,a\NN(y) + b\NN(x) +  \Tr(\ovx \ovC y) \mid
			 \ovC y + bx, a\ovy + \ovx\ovC, C).
	\end{equation*}
	Therefore, such an element preserves $\J_{10}^{abC}$ if and only if the following conditions hold for arbitrary $a,b$, and $C$:
	\begin{equation*}
		\left.
		\begin{array}{r@{\;}c@{\;}l}
			a\NN(y) + b\NN(x) + \Tr(\ovx \ovC y) & = & 0, \\
			\ovC y + bx & = & 0, \\
			a\ovy + \ovx\ovC & = & 0.
		\end{array}
		\right\}
	\end{equation*}
	In particular, if we take $C = 0$, $b = 1$ then we obtain $x = 0$, and when $C = 0$, $a = 1$, we get $y = 0$. 
	
	Let $T_2$ be the subgroup of $\SE_6(F)$ generated by the actions on $\J$ of $M_x''$ and $L_y'$ as $x,y$ range through $\OO$. 
	It can be shown that $T_2$ is isomorphic to $F^{16}$; this proof is rather similar to the proof $T_1 \cong F^{16}$. 
	Consider the $26$-dimensional space $\J_{26}^{abABC}$ spanned by the $17$-spaces corresponding to the white vectors in 
	$\J_{10}^{abC}$.  Let \mbox{$(a,b,c \mid A,B,C)$} be a white vector outside $\J_{26}^{abABC}$, that is, with $c \neq 0$. The whiteness conditions
	imply that such a vector has the form $(B\ovB / c, A\ovA/c, c \mid A, B, \ovB \ovA / c)$. $T_2$ acts sharply transitively on white points
	spanned by these vectors. Therefore, the full stabiliser of $\J_{10}^{abC}$ is indeed 
	$F^{16}\cn\Spin_{10}^+(F).F^{\times}$.
\end{proof}

Next, we investigate the orbits of the stabiliser of $\J_{10}^{abC}$ on white vectors and white points. First, we consider white vectors
in $\J_{10}^{abC}$. An arbitrary non-zero Albert vector $(a,b,0\mid 0,0,C)$ is white if and only if $a b - C\ovC = 0$.
Note that the stabiliser of $\J_{10}^{abC}$ acts on the [quotient] $10$-space as $\Omega_{10}^+(F)$, and is transitive on such vectors. 
Therefore, the stabiliser of $\J_{10}^{abC}$ is transitive on white vectors in
$\J_{10}^{abC}$, and on the white points spanned by them. 

Suppose now that $(a,b,c \mid A,B,C)$ is a white vector such that \mbox{$(c,A,B) \neq (0,0,0)$}. We consider two cases.
First, assume $c \neq 0$. The element $M_{-c^{-1} B}''$ maps our vector to $(a - c^{-1}\NN(B), b, c \mid A,0, C-c^{-1}\ovB \ovA)$, and since
the latter is white, we have \mbox{$a - c^{-1}\NN(B) = 0  = C - c^{-1}\ovB\ovA$}, so our new vector is of the form
$(0,b,c \mid A,0,0)$. Similarly, we act on it by $L_{-c^{-1} \ovA}$ to obtain $(0,b-c^{-1}\NN(A),c \mid 0,0,0)$, and since this vector is also 
white, we have $b - c^{-1} \NN(A) = 0$, so the resulting vector is of the form $(0,0,c\mid 0,0,0)$, $c \neq 0$. 

Second, we consider the case $c = 0$, so we start with $(a,b,0\mid A,B,C)$ where $(A,B) \neq (0,0)$. The whiteness conditions are
\begin{equation}
	\left.
	\begin{array}{r@{\;}c@{\;}l}
		A\ovA & = & 0, \\
		B\ovB & = & 0, \\
		C\ovC & = & ab, \\
		AB & = & 0, \\
		BC & = & a\ovA, \\
		CA & = & b\ovB. 
	\end{array}
	\right\}
\end{equation}
If $(a,b) = 0$, then we choose a suitable $M_x$, $M_x''$, $L_x$, or $L_x'$ to map our vector to a vector of the form
$(a',b',0\mid A',B',C')$ with $(a',b') \neq (0,0)$. Thus, we may assume $(a,b) \neq (0,0)$. 

We again distinguish two cases. 
If $a \neq 0$, then we act on $(a,b,0\mid A,B,C)$ by $M_{-a^{-1} C}$ to get $(a,0,0\mid 0,B,0)$. Next, we act by 
$M_y''$ to get \mbox{$(a + \Tr(\ovy B), 0, 0 \mid 0,B,0)$}, and choosing a suitable $y$ we obtain $(0,0,0\mid 0,B,0)$.

If $b \neq 0$, the action by $L_{-b^{-1} \ovC}$ maps $(a,b,0\mid A,B,C)$ to \mbox{$(0,b,0\mid A,0,0)$} and 
similarly acting by a suitable $L_y'$, we obtain $(0,0,0\mid A,0,0)$. Recall that the duality element $\delta$ preserves 
$\J_{10}^{abC}$, so it is enough to consider vectors $(0,0,0\mid 0,0,0)$ with $A \in \OO$. By choosing a copy of $\Omega_8^+(F)$ 
generated by the elements $P_u'$ with $u$ being an octonion of norm one, we can map $(0,0,0\mid A,0,0)$ to 
$(0,0,0\mid e_0, 0,0)$.

Finally, it is impossible to map $(a,b,c\mid A,B,C)$ with $c \neq 0$ to $(a,b,0\mid A,B,C)$ using any of the
elements $M_x$, $M_x''$, $L_x$, and $L_x'$, so these two belong to different orbits. In other words, we have shown the following 
result.

\begin{theorem}
	\label{theorem:1_white_space_stab_orbits}
	The stabiliser in $\SE_6(F)$ of\phantom{;} $\J_{10}^{abC}$ has three orbits on white points:
	\begin{enumerate}[(i)]
		\item images of $\langle (1,0,0 \mid 0,0,0) \rangle $ under the action of the stabiliser,
		\item images of $\langle (0,0,0\mid e_0,0,0) \rangle$ under the action of the stabiliser,
		\item images of $\langle (0,0,1 \mid 0,0,0) \rangle$ under the action of the stabiliser.
	\end{enumerate}
\end{theorem}

Now that we know the structure of the stabiliser of $\J_{10}^{abC}$ in $\SE_6(F)$, and moreover, we know the orbits of 
its action on white points, it is possible to figure out the joint stabilisers. The results are presented in a series of lemmas.

\begin{lemma}
	\label{lemma:1_stab_joint_1}
	The joint stabiliser in $\SE_6(F)$ of\phantom{;} $\J_{10}^{abC}$ and $\langle (1,0,0 \mid 0,0,0) \rangle$ is a subgroup
	of shape 
	\begin{equation}
		F^{16}\cn F^8\cn\Spin_8^+(F).F^{\times}.F^{\times}.
	\end{equation}
\end{lemma}

\begin{proof}
	The orbit of $\langle (1,0,0 \mid 0,0,0) \rangle$ under the action of the stabiliser of $\J_{10}^{abC}$ is the set of all white points
	spanned by the vectors in this $10$-space.  The
	normal subgroup 
\mbox{$F^{16} \trianglelefteqslant F^{16}\cn\Spin_{10}^+(F).F^{\times}$} acts trivially on $\J_{10}^{abC}$.  
Any white point $\langle (a,b,0 \mid 0,0,C) \rangle$ satisfies $ab - C\ovC = 0$,  so preserving a white point
	is equivalent to preserving an isotropic point in $\Spin_{10}^+(F).F^{\times}$ and thus we obtain the action of a group of shape $F^{16}\cn F^8\cn\Spin_8^+(F).F^{\times}.F^{\times}$.
\end{proof}

\begin{lemma}
	\label{lemma1_stab_joint_2}
	The joint stabiliser in $\SE_6(F)$ of\phantom{;} $\J_{10}^{abC}$ and $\langle (0,0,0 \mid e_0,0,0) \rangle$ is a
	 subgroup
	of shape 
	\begin{equation}
		F^{11}\cn F^{10}\cn\SL_5(F).F^{\times}.F^{\times}. 
	\end{equation}
\end{lemma}

\begin{proof}
	The $17$-space $U$ of $\langle (0,0,0 \mid e_0,0,0) \rangle$ 
	consists of the vectors 
	$(0,b,c \mid A,B,C)$ where $b,c \in F\cdot 1_{\OO}$, and
	\begin{equation*}
		\begin{array}{r@{\;}c@{\;}l}
			A & \in & \langle e_{-1}, e_{\ombb}, e_{\omega}, e_0, e_{-\omega}, e_{-\ombb}, e_1 \rangle, \\
			B & \in & \langle e_{\ombb}, e_{\omega}, e_{-0}, e_1 \rangle, \\
			C & \in & \langle e_{-1}, e_{-0}, e_{-\omega}, e_{-\ombb} \rangle.
		\end{array}
	\end{equation*}
	The joint stabiliser preserves this space and hence also its intersection with $\J_{10}^{abC}$, which is the 
	$5$-dimensional space
	$\langle (0,b,0 \mid 0,0,C) \rangle$ where $C \in  \langle e_{-1}, e_{-0}, e_{-\omega}, e_{-\ombb} \rangle$. Consider the actions of
	$M_x$, $L_x$, $M_x''$, and $L_x'$ on this intersection:
	\begin{equation*}
		\begin{array}{r@{\;}c@{\;}l}
			M_x : (0,b,0\mid 0,0,C) & \mapsto & (0,b+\Tr(\ovx C),0\mid 0,0,C), \\
			L_x : (0,b,0\mid 0,0,C) & \mapsto & (b\NN(x) + \Tr(C x),b,0\mid 0,0,C), \\
			M_x'' : (0,b,0\mid 0,0,C) & \mapsto & (0,b,0\mid 0,0,C), \\
			L_x' : (0,b,0\mid 0,0,C) & \mapsto & (0,b,0\mid 0,0,C).
		\end{array}
	\end{equation*}	 
	It follows that $M_x$ preserves the $5$-space for any $x \in \OO$, and $L_x$ does so if and only if $b\NN(x) + \Tr(Cx) = 0$, from 
	which we find $x \in \langle e_{-1}, e_{0}, e_{-\omega}, e_{-\ombb} \rangle$. Finally, $M_x''$ and $L_x'$ act trivially on the intersection.

We choose the following basis in $\J_{10}^{abC}$:
	\begin{equation*}
		\begin{array}{r@{\;}c@{\;}l}
			v_1 & = & (0,-1,0 \mid 0,0,0), \\
			v_2 & = & (0,0,0 \mid 0,0, e_{-1}), \\
			v_3 & = & (0,0,0 \mid 0,0, e_{-0}), \\
			v_4 & = & (0,0,0 \mid 0,0, e_{-\omega}), \\
			v_5 & = & (0,0,0 \mid 0,0, e_{-\ombb}),
		\end{array}\ \ 
		\begin{array}{r@{\;}c@{\;}l}
			w_1 & = & (1,0,0 \mid 0,0,0), \\
			w_2 & = & (0,0,0 \mid 0,0,e_1), \\
			w_3 & = & (0,0,0 \mid 0,0,e_0), \\
			w_4 & = & (0,0,0 \mid 0,0,e_{\omega}), \\
			w_5 & = & (0,0,0 \mid 0,0,e_{\ombb}). 
		\end{array}
	\end{equation*}
	The first five vectors form a basis of $U\cap\J_{10}^{abC}$. Note that the elements  $M_x$ with 
	\mbox{$x \in \langle e_{-1}, e_{-0}, e_{-\omega}, e_{-\ombb}\rangle$} act trivially on this $5$-space.  With respect to the basis $\{ v_1, \ldots, v_5 \}$
we obtain all the elementary transvections with non-zero 
off-diagonal entry in the first row or the first column by
taking $M_{\lambda e_i}$ for 
$i \in \{\ombb, \omega, 0, 1\}$ and $L_{\lambda e_i}$ where
$i \in \{-1, -0, -\omega, -\ombb\}$, which generate
$\SL_5(F)$.	
	
Now we are going back to the action on the $10$-space. 
With respect to the chosen basis the elements 
	$M_x$ with $x \in \langle e_{\ombb}, e_{\omega}, e_0, e_1 \rangle$ and 
	$L_x$ with $x \in \langle e_{-1}, e_0, e_{-\omega}, e_{-\ombb} \rangle$ act as 
	the matrices of the form 
	\begin{equation*}
		\left[
			\begin{array}{@{\;\;\;\;\;}c@{\;\;\;\;\;}|c}
				S & 0 \\ \hline
				0 & (S^{-1})^{\T}
			\end{array}
		\right],
	\end{equation*}
	Indeed, for example $L_{\lambda e_{-1}}$ has the following matrix form:
	\begin{equation*}
	\begin{bmatrix}[r]
	
		1 & \lambda & \cdot & \cdot & \cdot & \cdot & \cdot & \cdot & \cdot & \cdot  \\
		 \cdot & \phantom{-}1 & \cdot & \cdot & \cdot & \cdot & \cdot & \cdot & \cdot & \cdot  \\
		 \cdot & \cdot & \phantom{-}1 &  \cdot & \cdot & \cdot & \cdot & \cdot & \cdot & \cdot \\
		  \cdot & \cdot & \cdot & \phantom{-}1 & \cdot & \cdot & \cdot & \cdot & \cdot & \cdot \\
		  \cdot & \cdot & \cdot & \cdot & \phantom{-}1 & \cdot & \cdot & \cdot & \cdot & \cdot \\
		  \cdot & \cdot & \cdot & \cdot & \cdot & \phantom{-}1 &  \cdot & \cdot & \cdot & \cdot \\
		  \cdot & \cdot & \cdot & \cdot & \cdot & -\lambda & \phantom{-}1 &  \cdot & \cdot & \cdot  \\
		  \cdot & \cdot & \cdot & \cdot & \cdot & \cdot & \cdot & \phantom{-}1 &  \cdot & \cdot \\
		  \cdot & \cdot & \cdot & \cdot & \cdot & \cdot & \cdot & \cdot & \phantom{-}1 & \cdot \\
		  \cdot & \cdot & \cdot & \cdot & \cdot & \cdot & \cdot & \cdot & \cdot & \phantom{-}1  
	
	\end{bmatrix}.
	\end{equation*}
The map $S \mapsto (S^{-1})^{\T}$ is an automorphism of 
$\SL_5(F)$ of duality type. Recall that the action 
generated by the
 elements $M_x$ and $L_x$, where $x$ ranges through $\OO$, 
 on the whole $27$-space $\J$ has a non-trivial kernel,
 generated by the element $P_{-1}$, when the characteristic of 
 the field is not $2$. In the current situation we do not
 get this central involution as those $M_x$ and $L_x$ which 
 generate $\SL_5(F)$ fix the vector $(0,0,0 \mid e_0,0,0)$
 in the $16$-space $\J_{16}^{AB}$.

The elements $M_x$ with $x \in \langle e_{-1}, e_{-0}, e_{-\omega}, e_{-\ombb} \rangle$ act trivially on $U\cap \J_{10}^{abC}$, but
	on the whole $10$-space they act as
	\begin{equation*}
		\left[
			\begin{array}{c|c}
				\II_5 & 0 \\ \hline
				A & \II_5
			\end{array}
		\right],
	\end{equation*}
	where $A$ is a $5 \times 5$ matrix. For instance, the element $M_{\lambda e_{-1}}$ is represented by the matrix
	\begin{equation*}
	\begin{bmatrix}[r]
	
		1 & \cdot & \cdot & \cdot & \cdot & \cdot & \cdot & \cdot & \cdot & \cdot  \\
		 \cdot & 1 & \cdot & \cdot & \cdot & \cdot & \cdot & \cdot & \cdot & \cdot  \\
		 \cdot & \cdot &1 &  \cdot & \cdot & \cdot & \cdot & \cdot & \cdot & \cdot \\
		  \cdot & \cdot & \cdot & 1 & \cdot & \cdot & \cdot & \cdot & \cdot & \cdot \\
		  \cdot & \cdot & \cdot & \cdot & 1 & \cdot & \cdot & \cdot & \cdot & \cdot \\
		  \cdot & \lambda  & \cdot & \cdot & \cdot & 1 &  \cdot & \cdot & \cdot & \cdot \\
		  -\lambda & \cdot & \cdot & \cdot & \cdot & \cdot & 1 &  \cdot & \cdot & \cdot  \\
		  \cdot & \cdot & \cdot & \cdot & \cdot & \cdot & \cdot & 1 &  \cdot & \cdot \\
		  \cdot & \cdot & \cdot & \cdot & \cdot & \cdot & \cdot & \cdot & 1 & \cdot \\
		  \cdot & \cdot & \cdot & \cdot & \cdot & \cdot & \cdot & \cdot & \cdot & 1  
	
	\end{bmatrix}.
	\end{equation*}
With respect to our chosen basis, the inner product in 
	$\J_{10}^{abC}$ has the form
	\begin{equation*}
		\left[
			\begin{array}{c|c}
				0 & -\II_5 \\ \hline
				-\II_5 & 0
			\end{array}
		\right],
	\end{equation*}
	and so we can derive the conditions on the matrix $A$:
	\begin{equation*}
		\left[
			\begin{array}{c|c}
				\II_5 & 0 \\ \hline
				A & \II_5
			\end{array}
		\right]
		\left[
			\begin{array}{c|c}
				0 & -\II_5 \\ \hline
				-\II_5 & 0
			\end{array}
		\right]
		\left[
			\begin{array}{c|c}
				\II_5 & A^{\T} \\ \hline
				0 & \II_5
			\end{array}
		\right] = 
		\left[
			\begin{array}{@{\;\;\;\;\;\;\;}c@{\;\;\;\;\;\;\;}|c}
				 0 & -\II_5 \\ \hline
				-\II_5 & -(A+A^{\T})
			\end{array}
		\right].
	\end{equation*}
	It follows that $A + A^{\T} = 0$. In characteristic other than $2$ this also implies that the diagonal entries in $A$ are zero. 
	In characteristic $2$ we need to consider the quadratic form to obtain the same result.

Let $g$ be an element of the joint
stabiliser represented by the matrix
	\begin{equation*}
		\left[
			\begin{array}{c|c}
				\II_5 & 0 \\ \hline
				A & \II_5
			\end{array}
		\right],
	\end{equation*}
	with respect to the chosen basis.  We have $v_i^g = v_i$ and \mbox{$w_i^g = w_i + \sum_j A_{i j} v_j$} for all $i$ such that
	$1 \leqslant i \leqslant 5$. It follows that 
\begin{multline*}
		\QQ_{10}(w_i^g) = \QQ_{10}(w_i + A_{i 1} v_1 + \ldots + A_{i 5} v_5) \\
		    =\QQ_{10}(w_i) + \inner{w_i}{A_{i 1} v_1 + \ldots + A_{i 5} v_5} + \QQ_{10}(A_{i 1} v_1 + \ldots + A_{i 5} v_5) \\
		    = \inner{w_i}{A_{i 1} v_1 + \ldots + A_{i 5} v_5}  = -A_{ii},
\end{multline*}
	so we indeed have $A_{i i} = 0$ for $1 \leqslant i \leqslant 5$. 
	
	Regardless of characteristic, such matrices $A$ reside
	in a $10$-dimensional $F$-space. 
	In fact, all such matrices
	span a $10$-dimensional $F$-space $W$, and the spanning set
	can be obtained, for example, by taking relevant 
	$5 \times 5$ blocks in the $10 \times 10$ matrices, 
	representing the following elements:
\begin{equation*}
	\begin{array}{c@{\ \ \ \ \ \ \ \ \ \ \ }c}
M_{e_{-1}},&
	(L_{e_{-\omega}})^{-1}\cdot M_{e_{-1}}\cdot L_{e_{-\omega}}, \\

M_{e_{-0}},&
	(L_{e_{-\ombb}})^{-1}\cdot M_{e_{-1}}\cdot L_{e_{-\ombb}}, \\
	
M_{e_{-\omega}},&
	(L_{e_{-\omega}})^{-1}\cdot M_{e_{-0}}\cdot L_{e_{-\omega}}, \\
	
M_{e_{-\ombb}},&
	(L_{e_{-\ombb}})^{-1}\cdot M_{e_{-0}}\cdot L_{e_{-\ombb}}, \\

(L_{e_{0}})^{-1}\cdot M_{e_{-1}}\cdot L_{e_{0}}, &

(L_{e_{-\ombb}})^{-1}\cdot M_{e_{-\ombb}}\cdot L_{e_{-\ombb}}.
	\end{array}
\end{equation*}

Now, $V$, spanned by $v_1, \ldots, v_5$, 
is an $FG$-module with $G = \SL_5(F)$. We identify the 
$25$-space spanned by the $5 \times 5$ matrices
$E_{i j}$ (having $1$ as the $(i,j)$-entry and zeroes in all
other positions) with the tensor square $V \otimes V$ by
establishing an isomorphism 
$\varphi: E_{i j} \mapsto v_i \otimes v_j$
and extending by linearity. An arbitrary element $g$ of 
$\SL_5(F)$ acts on these matrices via the map 
$E_{i j} \mapsto S^{\T} E_{i j} S$ 
for some $5 \times 5$ matrix $S$,
and we have 
\begin{multline*}
	(v_i \otimes v_j)^g = \varphi(E_{i j}^g) = 
		\varphi\left(\sum_{r,s} S_{i r} S_{j s} E_{r s}\right) 
	= \sum_{r,s} S_{i r} S_{j s} (v_r \otimes v_s) \\ =
	\sum_{r,s} (S_{i r} v_r) \otimes (S_{j s} v_s) = 
	(e_i S) \otimes (e_j S).
\end{multline*}
Thus, $\varphi$ is indeed an isomorphism of $FG$-modules. 
The elements $E_{i j} - E_{j i}$ form a spanning set for the
$10$-space $W$, defined earlier, and so the image of 
$W$ under the isomorphism $\varphi$ is $\bigwedge^2 (V)$,
as a submodule of $V \otimes V$. Thus, we have
	the action of $\SL_5(F)$ on $\bigwedge^2(V) \cong F^{10}$.
	
The elements $P_u$ and $P_u'$ act on the white vector $(0,0,0 \mid e_0,0,0)$ in the following way:
\begin{equation*}
	\begin{array}{r@{\;}c@{\;}l}
		P_u : (0,0,0 \mid e_0, 0,0) & \mapsto & (0,0,0 \mid u e_0 \NN(u)^{-1}, 0,0), \\
		P_u' : (0,0,0 \mid e_0,0,0) & \mapsto & (0,0,0 \mid \ovu e_0 \ovu \NN(u)^{-1}, 0,0).
	\end{array}
\end{equation*}
Taking $u = \lambda \cdot 1_{\OO}$ where $\lambda \in F$ for both $P_u$ and $P_u'$, we obtain the action
of the group of shape $F^{\times}.F^{\times}$, which preserves the white point $\langle (0,0,0 \mid e_0,0,0) \rangle$.

So far we have shown that the joint stabiliser is  no bigger than the group of shape  \mbox{$F^{16}\cn F^{10}\cn \SL_5(F).F^{\times}.F^{\times}$}.
We notice that $M_x''$ with $x \in \langle e_{-1}, e_0, e_{-\omega}, e_{-\ombb} \rangle$ and $L_y'$ with 
$y \in \langle e_{-1}, e_{\ombb}, e_{\omega}, e_{-0}, e_{-\omega}, e_{-\ombb}, e_1\rangle$ preserve
the white point $\langle (0,0,0 \mid e_0,0,0) \rangle$. These elements generate an abelian group isomorphic to $F^{11}$. Let now
$M_x'' \cdot L_x'$ be a general element in $F^{16}$. It sends $(0,0,0 \mid e_0, 0,0)$ to $(0,\Tr(e_0 y),0 \mid e_0, 0, \ovx e_{-0})$,
so we get the following conditions on $x$ and $y$:
\begin{equation*}
	\left.
		\begin{array}{r@{\;}c@{\;}l}
			\ovx e_{-0} & = & 0, \\
			\Tr(e_0 y) & = & 0.
		\end{array}
	\right\}
\end{equation*}
The first condition implies that $\ovx e_0 \in \langle e_{\ombb}, e_{\omega}, e_0, e_1 \rangle$, so
$x \in \langle e_{-1}, e_0, e_{-\omega}, e_{-\ombb} \rangle$. From the second condition we derive
$y \in \langle e_{-1}, e_{\ombb}, e_{\omega}, e_{-0}, e_{-\omega}, e_{-\ombb}, e_1\rangle$, 
so indeed the joint stabiliser is $F^{11}\cn F^{10} \cn \SL_5(F).F^{\times}.F^{\times}$. 
\end{proof}

We have already seen the proof of the following lemma (Theorem \ref{theorem:1_white_space_stab}), so we just state the result here.
\begin{lemma}
	\label{lemma1_stab_joint_3}
	The joint stabiliser in $\SE_6(F)$ of\phantom{;} $\J_{10}^{abC}$ and $\langle (0,0,1 \mid 0,0,0) \rangle$ is a subgroup
	of shape 
	\begin{equation}
		\Spin_{10}^+(F).F^{\times}.
	\end{equation}
\end{lemma}

\section{Conclusions}

We have managed to obtain a self-contained construction of a group of type $\EE_6$ over an
arbitrary field $F$. However, we want to point out that there is a space for future research.
For example, the main result in Section \ref{section:stabiliser} depends on the fact 
that the underlying octonion algebra $\OO$ is split. This completely covers the possibilities
in case $F = \Fq$, but it seems to be quite interesting to understand to what extent it is
possible to generalise our result in the case when a non-split octonion algebra exists. 
The main problem here is to be able to tell whether the actions of the matrices $M_x$ and 
$L_x$ on $\J_{10}^{abC}$ generate $\Omega( \J_{10}^{abC}, \QQ_{10} )$. At this stage it is
possible to prove the following proposition.

\begin{proposition}
	The actions of the elements $M_x$ and $L_x$ on $\J_{10}^{abC}$ where $x$ ranges through
	a non-split octonion algebra $\OO$, generate a subgroup of $\Omega( \J_{10}^{abC},
	\QQ_{10} )$.
\end{proposition}

\begin{proof}
	It is straightforward to verify that the image of $(a,b,0\mid 0,0,C)$ under the action
	of $M_x$ coincides with the image under the reflexion in $(0,\NN(x),0 \mid 0,0,\ovx)$ followed
	by a reflexion in $(0,0,0 \mid 0,0,\ovx)$. The norm of both vectors is $\NN(x)$, so we conclude
	that $M_x$ acts as an element of $\Omega( \J_{10}^{abC}, \QQ_{10} )$. 
	
	For the element $L_x$ we take the vectors $(\NN(x),0,0\mid 0,0,x)$ and $(0,0,0\mid 0,0,x)$ to
	obtain the same conclusion. 
\end{proof}

Now suppose that $V$ is a vector space over $F$ with a quadratic form $Q$, 
such that $V = \langle e,f \rangle  \oplus W$, where $(e, f)$ is a hyperbolic pair and 
$W = \langle e, f \rangle^{\perp}$. Consider an element $g$ in $\CGO(V,Q)$ which scales of $Q$ 
by some $\lambda \neq 0$. Then
$V = \langle e^g, f^g \rangle \oplus W^g$ and $\langle e^g, f^g \rangle$ is
isometric to $\langle e, f \rangle$. Therefore, $W^g$ is isometric to 
$W$, and so there exists an isometry $h$ in $\GO(V,Q)$ such that 
$\langle e^g, f^g \rangle^h = \langle e, f \rangle$. It follows that 
$(W^g)^h = W$, and $gh$ is a $\lambda$-scaling of $Q$ which fixes
$\langle e, f \rangle$ and $W$. Hence, $gh$ is a $\lambda$-scaling of $Q_W$. 

Consider a $\lambda$-similarity on $\OO = \OO_F$ that sends $1_{\OO}$ to some $u \in \OO$. 
Then it gives rise to an element in the stabiliser of a white point which scales $\QQ_8$ by
$\NN(u)$. In other words, we have shown the following.

\begin{proposition}
	If $\OO$ is an arbitrary octonion algebra over $F$, then the elements in the stabiliser
	of a white point can only scale a white vector by $\lambda$, where $\lambda \in F$ is
	such that there exists $u \in \OO$ with $\NN(u) = \lambda$.  
\end{proposition}

It is easy to check that all such scalings are possible. For example, the elements 
$P_{u^{-1}}'$, defined in (\ref{eq:puprime}), do the job.

%\affiliationone{
%   School of Mathematical Sciences,\\Queen Mary, University of London,\\
%Mile End Road,\\ London E1 4NS,\\ U.K.}
%  \email{e.stepanov@qmul.ac.uk}

\end{document}